\documentclass[12pt,reqno]{article}

\usepackage[usenames]{color}
\usepackage{amssymb}
\usepackage{graphicx}
\usepackage{amscd}
\usepackage{xcolor}
\usepackage[colorlinks=true,
linkcolor=webgreen,
filecolor=webbrown,
citecolor=webgreen]{hyperref}

\definecolor{webgreen}{rgb}{0,.5,0}
\definecolor{webbrown}{rgb}{.6,0,0}

\usepackage{color}
\usepackage{fullpage}
\usepackage{float}

\usepackage{graphics,amsmath,amssymb}
\usepackage{amsthm}
\usepackage{amsfonts}
\usepackage{latexsym}
\usepackage{epsf}

\usepackage[noblocks]{authblk}
\usepackage{tikz}

\DeclareMathOperator{\MMP}{MMP}

\theoremstyle{plain}
\newtheorem{theorem}{Theorem}
\newtheorem{corollary}[theorem]{Corollary}
\newtheorem{lemma}[theorem]{Lemma}

\newtheorem{observation}[theorem]{Observation}

\theoremstyle{definition}
\newtheorem{definition}[theorem]{Definition}

\theoremstyle{remark}

\newcommand{\asc}{{\rm asc}}
\newcommand{\dist}{{\rm dist}}
\newcommand{\maxx}{{\rm maxx}}
\newcommand{\maxim}{{\rm maxim}}
\newcommand{\zeros}{{\rm zeros}}
\newcommand{\repeats}{{\rm repeats}}
\newcommand{\last}{{\rm last}}
\newcommand{\bottom}{{\rm bottom}}
\newcommand{\exc}{{\rm exc}}
\newcommand{\tp}{{\rm top}}

\newcommand{\Sn}{\mathbf{S}}
\newcommand{\I}{\mathbf{I}}

\title{Patterns in Inversion Sequences II:  Inversion Sequences Avoiding Triples of Relations}
\author{Megan A. Martinez\thanks{Department of Mathematics, Ithaca College, Ithaca, NY 14850, {\tt mmartinez@ithaca.edu}} \ \ 
 and Carla D. Savage\thanks{Department of Computer Science, North Carolina State University, Raleigh, NC 27695-8206, {\tt savage@ncsu.edu}}  
 }

\begin{document}

\begin{center}
\vskip 1cm{\Large\bf 
Patterns in Inversion Sequences II:  \\
\vskip .1in
Inversion Sequences Avoiding Triples of Relations
}
\vskip 1cm
\end{center}

\begin{center}
Megan Martinez\\
Department of Mathematics\\
Ithaca College\\
Ithaca, NY 14850\\
USA \\
\href{mailto:mmartinez@ithaca.edu}{\tt mmartinez@ithaca.edu} \\
\ \\
Carla Savage\\
Department of Computer Science\\
North Carolina State University\\
Raleigh, NC 27695\\
USA \\
\href{mailto:savage@ncsu.edu}{\tt savage@ncsu.edu }\\
\ \\
\end{center}

\vskip .2in

\begin{abstract}
Inversion sequences of length~$n$, $\mathbf{I}_n$, are integer sequences $(e_1, \ldots, e_n)$ with  $0 \leq e_i < i$ for each $i$. The study of patterns in inversion sequences was initiated recently by Mansour-Shattuck and Corteel-Martinez-Savage-Weselcouch through a systematic study of inversion sequences avoiding words of length 3.  We continue this investigation by reframing the notion of a length-3 pattern from a ``word of length 3,''  $w_1w_2w_3$, to a  ``triple of binary relations,'' $(\rho_1,\rho_2,\rho_3)$, and consider the set $\mathbf{I}_n(\rho_1,\rho_2,\rho_3)$ consisting of those $e \in \mathbf{I}_n$ with no $i < j < k$ such that
$e_i \rho_1 e_j$,
$e_j \rho_2 e_k$,  and
$e_i \rho_3 e_k$. We show
 that ``avoiding a triple of relations''  can characterize inversion sequences with a variety  of  monotonicity or unimodality conditions, or  with multiplicity constraints on the elements.  We uncover several interesting enumeration results and relate pattern avoiding inversion sequences to familiar combinatorial families.  
 We highlight open questions about the relationship between pattern avoiding inversion sequences and a variety of classes of pattern avoiding permutations. For several combinatorial sequences, pattern avoiding inversion sequences provide a simpler interpretation than otherwise known.

\end{abstract}

\section{Introduction}

Pattern avoiding permutations  have been studied  extensively for  their connections in computer science, biology, and other fields of mathematics.  Within combinatorics they have proven their usefulness, providing an interpretation that relates a vast array of combinatorial structures.  
See the comprehensive survey of Kitaev \cite{kitaev}.

The notion of pattern avoidance in inversion sequences was introduced by Corteel, Martinez, Savage, and Weselcouch \cite{paisI} and Mansour and Shattuck \cite{Mansour}.
An {\em inversion sequence} is an integer  sequence $(e_1,e_2,\ldots,e_n)$  satisfying $0 \leq e_i<i$ for all $i \in [n]=\{1,2, \ldots, n\}$. There is a natural bijection   $\Theta: \Sn_n \rightarrow \I_n$ from $\Sn_n$, the set of permutations of $[n]$, to  $\I_n$, the set of inversion sequences of length $n$. 
Under this bijection, $e= \Theta(\pi)$  is obtained from a  permutation $\pi=\pi_1\pi_2\ldots \pi_n \in \Sn_n$ by setting $e_i = |\{j \ | \ j < i \ {\rm and} \ \pi_j > \pi_i \}|$.  
  
 The encoding of permutations as inversion sequences 
 suggests that it could be illuminating to study patterns in inversion sequences in the same way that patterns have been studied in permutations.  The paper of Corteel, et al.\ \cite{paisI} focused on the enumeration of inversion sequences that avoid words of length three and the paper of Mansour and Shattuck
 \cite{Mansour} targets permutations of length 3.  For example, the inversion sequences $e \in \I_n$ that avoid the pattern 021 are those with no
$i < j < k$ such that $e_i < e_j > e_k$ and $e_i < e_k$.  We denote these by $\I_n(021)$.
Similarly, $\I_n(010)$ is the set of $e \in \I_n$ with no
$i < j < k$ such that $e_i < e_j > e_k$ and $e_i = e_k$.
The results  in each of these works \cite{paisI,Mansour} related pattern avoidance in inversion sequences to a number of well-known combinatorial sequences including the Fibonacci numbers, Bell numbers, large Schr\"oder numbers, and Euler up/down numbers. They also gave rise to natural sequences that previously had not appeared in the On-Line Encyclopedia of Integer Sequences (OEIS) \cite{Sloane}.

In this paper we reframe the notion of a length-3 pattern from a word of length 3 to a triple of binary relations.  For a  fixed triple of binary relations $(\rho_1,\rho_2,\rho_3)$, we study the set $\I_n(\rho_1,\rho_2,\rho_3)$ consisting of those $e \in \I_n$ with no $i < j < k$ such that
$e_i \rho_1 e_j$,
$e_j \rho_2 e_k$,  and
$e_i \rho_3 e_k$.
For example, $\I_n(<,>,<) = \I_n(021)$ and  $\I_n(<,>,=) = \I_n(010)$.  
As another example, $\I_n(\geq,\geq,\geq)$ is the set of inversion sequences in $\I_n$ with no weakly decreasing subsequence of length 3, i.e., those inversion sequences avoiding all of the words in the set $\{000,100,110,210\}$.

Table \ref{interesting} illustrates  that avoiding a single triple of relations can characterize inversion sequences with a variety  of natural monotonicity or unimodality conditions, or  with multiplicity constraints on the appearance of elements in the inversion sequence.  Moreover, as will be seen, inversion sequences avoiding a single triple of relations can provide  realizations of certain combinatorial sequences that are simpler than known realizations as pattern avoiding permutations or other combinatorial structures.

For this project, we considered all triples of relations in the set $\{<,>,\leq,\geq,=, \neq,-\}^3$.  The relation ``$-$'' on a set $S$ is all of $S \times S$; that is, $x $  ``$-$''  $y$ for all $x,y \in S$.  
There are 343 possible triples of relations (patterns).  For each pattern, $(\rho_1, \rho_2, \rho_3)$, we can consider the avoidance set, $\I_n(\rho_1, \rho_2, \rho_3)$, or the avoidance sequence $|\I_1(\rho_1, \rho_2, \rho_3)|, |\I_2(\rho_1, \rho_2, \rho_3)|, |\I_3(\rho_1, \rho_2, \rho_3)|, \ldots$. We say two patterns are {\em equivalent} if they give rise to the same avoidance sets and two patterns are {\em Wilf-equivalent} if they yield the same avoidance sequence.
For example, patterns $(\geq,\geq,\geq)$ and $(\geq,\geq,-)$ are equivalent, whereas patterns $(\geq,\geq,-)$ and $(-,\leq,\geq)$ are inequivalent, but Wilf equivalent (see Section \ref{section:772}).

The 343 patterns partition into 98 equivalence classes of patterns.  Additionally, we conjecture that there are 63 Wilf-equivalence classes.  In this paper, we enumerate a number of avoidance sets either directly or by relating them to familiar combinatorial structures.  These relationships establish Wilf-equivalence between a number of inequivalent patterns.  However, in many cases, even where enumeration is elusive, Wilf-equivalence can be proved via a bijection.  This paper presents the results we have been able to prove, documents what has not yet been settled, and highlights the most intriguing open questions.

We uncovered several interesting enumeration results beyond those discovered in the initial introduction of patterns in inversion sequences \cite{paisI,Mansour}.
For example,  the inversion sequences $e$ with no $i < j < k$ such that $ e_i = e_j \leq e_k$ are counted by the Fibonacci numbers (as are, e.g., permutations avoiding the pair (321, 3412)).
Inversion sequences with no $i < j < k$ such that $e_i < e_j \leq e_k$ are counted by powers of two (as are, e.g., permutations avoiding (213, 312)).
Inversion sequences avoiding $(- ,\neq, =)$ are counted by the Bell numbers; inversion sequences avoiding 
$(\geq, -, >)$ are counted by the large   Schr\"oder numbers;  and inversion sequences avoiding  $(- ,\geq, <)$ are counted by the Catalan numbers.
There are the same number of inversion sequences in $\I_n$ avoiding $(\neq,\neq,\neq)$ as there are Grassmannian permutations in $\Sn_n$.
$\I_n(\neq,<,\leq)$  is counted by the number of 321-avoiding separable permutations in $\Sn_n$.
$\I_n(\neq,<,\neq)$ has the same number of elements  as
 the set of permutations in $\Sn_n$ avoiding both of the patterns 321 and 2143.

Since an earlier draft of this paper was posted \cite{PAISII}, several conjectures we made have been proven. Inversion sequences avoiding $e_i > e_j \geq e_k$ were shown to be counted by the ``semi-Baxter sequence'' as defined by Bouvel, et al.\ \cite{BouvelGuerriniEtAl}. It was also shown that the semi-Baxter sequence enumerates plane permutations, which are defined as those permutations avoiding the barred pattern $21\bar{3}54$ or, equivalently, the vincular pattern $2 \underline{14} 3$. Inversion sequences avoiding $e_i \geq e_j \geq e_k$ were shown to have the same counting sequence as set partitions avoiding enhanced 3-crossings. This was shown using the obstinate kernel method by Lin \cite{Lin} and by a bijection that utilizes 0,1 fillings of Ferrers shapes by Yan \cite{Yan}. Additionally, the set $\I_n(\geq,\geq,>)$ was shown to have the same counting sequence as the Baxter permutations through use of the obstinate kernel method by Kim and Lin \cite{KimLin}. A number of open questions concerning inversion sequences still remain, and these will be highlighted throughout this paper.

Of the 63 conjectured Wilf-equivalence classes,  five classes are counted by sequences that are ultimately constant.  In the remaining 58 classes, 30 have counting sequences that appear to match already existing sequences in the OEIS. The remaining 28 resulted in new entries in the OEIS.

In Sections \ref{section:7} through \ref{section:5040}  we present our results and conjectures
 for the 30 Wilf classes of pattern-avoiding inversion sequences that (appear to) match sequences in the OEIS.  
Table \ref{other triples} gives an overview.   Even for patterns with a ``no'' in this table, we are able to prove some Wilf-equivalence results.

For the patterns whose counting sequence did not originally match a sequence in the OEIS we have some limited
  results on Wilf-equivalence and counting.  Table \ref{questions} gives an overview of the patterns in these 28 Wilf classes in addition to their newly created OEIS number.  Our results and conjectures for a few of these patterns are presented in Section 3.

In Tables \ref{interesting}, \ref{other triples}, and \ref{questions}, each row represents an equivalence class of patterns whose identifier is given in the last column.
A Wilf class of patterns is identified by the number $a_7$, the number of inversion sequences of length 7 avoiding a pattern in the class.  Within a Wilf class, equivalence classes are labeled A,B,C, etc.
So, for example,  there are three equivalence classes of patterns counted by the Catalan numbers and these classes are labeled
429A, 429B, and 429C.

This paper is intended to provide a comprehensive overview of the enumeration sequences for a large number of classes of pattern avoiding inversion sequences. Many classes can be enumerated through basic counting techniques; in these situations, we have omitted the relevant proofs for the sake of brevity.

In the remainder of this section we give some definitions that will be needed throughout this paper, keeping the number of definitions to a minimum so that subsections can be somewhat self-contained. 
Finally, for completeness, we list the triples of relations whose avoidance sequences are ultimately constant.

\subsection{Encodings of permutations}

We compare  $\Theta$ with a few other common encodings of permutations mentioned later in this paper:  Lehmer codes, and $invcodes$, which are reverse Lehmer codes.

For a sequence $t=(t_1, \ldots, t_n)$, let $t^R=(t_n, t_{n-1}, \ldots, t_1)$, and, for a set of sequences $T$, let
$T^R=\{ t^R \ | \ t \in T\}$.
For a permutation $\pi=(\pi_1, \ldots, \pi_n) \in \Sn_n$, let $\pi^C= (n+1-\pi_1, \ldots, n+1-\pi_n)$, and, for a set of permutations $P$, let $P^C = \{ \pi^C \ | \ \pi \in P\}$.  We use the following encodings.

\begin{itemize}

\item Define $\Theta: \Sn_n \rightarrow \I_n$ such that $e = \Theta(\pi)$ if and only if $e_i = |\{j \ | \ j < i \ {\rm and} \ \pi_j > \pi_i \}|$.

\item Define $L: \Sn_n \rightarrow \I_n^R$ such that $ e = L(\pi)$ if and only if $e_i = |\{j \ | \ j > i \ {\rm and} \ \pi_j < \pi_i \}|$.

\item Define $invcode:  \Sn_n \rightarrow \I_n$ such that $e=invcode(\pi)$ if and only if $e^R = L(\pi)$.
\end{itemize}
 
Note that $invcode(\pi)=e$ if and only if $e=\Theta((\pi^{C})^{R})$.   Additionally, notice that if $\Theta(\pi)=e$, then $i$ is a {\em descent} of $\pi$ (that is, $\pi_i > \pi_{i+1}$), if and only if $i$ is an {\em ascent} of $e$ (that is, $e_i < e_{i+1}$).

We will make use of another encoding, $\phi$, in Sections \ref{section:1806} and \ref{section:3720}.

\subsection{Operations on inversion sequences}

For $e = (e_1, e_2, \ldots, e_n) \in \I_n$ and any integer $t$, define
$\sigma_t(e)=(e_1',e_2', \ldots, e_n')$, where $e_i'=e_i$ if $e_i=0$, and $e_i'=e_i+t$ otherwise.  So, $\sigma_t$ adds $t$ to the nonzero elements of a sequence (notice that $t$ could be negative).

Concatenation  is used to add an element to the beginning or end of an inversion sequence.
For $e = (e_1, e_2, \ldots, e_n) \in \I_n$, we have $0 \cdot e = (0,e_1,e_2, \ldots, e_n) \in \I_{n+1}$, and, if $0 \leq i \leq n$, we have $e \cdot i = (e_1, e_2, \ldots, e_n,i) \in \I_{n+1}$.  More generally, $x \cdot y$ denotes the concatenation of two sequences or two words $x,y$.

\subsection{Statistics on inversion sequences}

In several cases, statistics on inversion sequences helped to prove or refine the results in Table \ref{other triples} and make connections with statistics on other combinatorial families.  The statistics used in this paper are defined below for an inversion sequence $e \in \I_n$.
\begin{eqnarray*}
\asc(e) & = & |\{i \in [n-1] \  | \  e_i < e_{i+1} \}|\\
\zeros(e) & = & |\{i \in [n] \  |  \ e_i =0 \}|\\
\dist(e) & = & |\{e_1, e_2, \ldots, e_n \}|\\
\repeats(e) & = & |\{i \in [n-1] \ | \ e_i \in \{e_{i+1}, \ldots, e_n\} \} |= n - \dist(e) \\
\maxim(e) & = & |\{i \in [n] \  |  \ e_i =i-1 \}|\\
\maxx(e) & = & \max \{e_1, e_2, \ldots, e_n \}\\
\last(e) & = & e_n.
\end{eqnarray*}
These statistics are, respectively, the number of ascents, the number of zeros, the number of distinct elements, the number of repeats, the number of maximal elements, the maximum element, and the last element of $e$.

\subsection{Ultimately constant avoidance sequences}

It can be easily checked that the following are the equivalence classes whose avoidance sequences are eventually constant.
\begin{tabbing}
xxxxxxxxxxxxxxx\=xxxxxxxxxxxxxxxxx\= \kill

\> pattern \> avoidance sequence\\
\\
\>$(-,-,-)$
\> $1,2,0,0,0,0, \ldots$
\\
\>$(\leq,\leq,- )$
\> $1,2,1,0,0,0,\ldots$
\\
\>$(-,-,\neq)$
\> $1,2,2,1,1,1\ldots$
\\
\>$(-,-,<)$
\> $1,2,2,2,2,2,\ldots$
\\
\>$(-,\neq,-)$
\> $1,2,2,2,2,2,\ldots$
\\
\>$(-,\neq,\neq)$
\> $1,2,3,3,3,3,\ldots$
\\
\end{tabbing}

\begin{table}[]
 \hrule
 \centering
\resizebox{1\textwidth}{!}{
\begin{tabular}{lllll}
&&&&\\

Inversion sequences $e$ satisfying: & 
are those with no &
and are &
$a_7$, &
Sect.\\

&
$i < j < k$ such that: &
counted by: &
equiv class&
\\

Monotonicity constraints:&&&&\\

$e_1=e_2=\ldots=e_{n-1}$ & $e_i \neq e_j$ & $n$ & 7,D & \ref{section:7}\\

$\exists t$: $e_1 = \ldots = e_t  \leq e_{t+1} = e_{t+2}=  \ldots = e_n$ &   $e_i < e_j \neq e_k$ & $1+n(n-1)/2$ & 22,A  & \ref{section:22} \\

$\exists t$: $e_1 = \ldots = e_t < e_{t+1} < e_{t+2} < \ldots < e_n$
& 
$e_i<e_j \geq e_k$
& $2^{n-1}$  & 64,C  & \ref{section:64}\\

$e_1 \leq e_2  \leq  \ldots \leq e_{n-1}$
&
$e_i> e_j$
& Binom$(2n-2,n-1)$ & 924  & \ref{section:924}\\

$e_1 \leq e_2  \leq  \ldots \leq e_{n-1} \leq e_n$
&
$e_j > e_k$
& Catalan number $C_n$ & 429,A  & \ref{section:429}\\

$e_1 \leq e_2  < e_3 <  \ldots < e_n$
&
$e_j \geq e_k$
& $n$ &7,B  & \ref{section:7}\\

$e_1 < e_2   <  \ldots < e_{n-1}$
&
$e_i = e_j$
& $n$ &7,C  & \ref{section:7}\\

$e_2 \geq e_3 \geq \ldots \geq e_n$
&
$e_j < e_k$
& $n$ & 7,A  & \ref{section:7}\\

\\

Unimodality constraints:\\
\\
$\exists t$: $e_1 = \ldots = e_t \leq e_{t+1} \geq  0=  \ldots =0$ 
&  
$e_i<  e_j$, $e_i < e_k$
& $1+n(n-1)/2$ &  22,B  & \ref{section:22}\\

$\exists t$: $e_1 = \ldots = e_t \leq e_{t+1} \geq e_{t+2} \geq \ldots \geq e_n$
&  
$e_i\neq e_j < e_k$
& Grassmannian perms & 121,A  & \ref{section:121}\\

$\exists t$: $e_1 = \ldots = e_t < e_{t+1} > e_{t+2} > \ldots > e_n$ 
&  
$e_ i\neq e_j\leq e_k$
& $F_{n+2}-1$ & 33,A  & \ref{section:33} \\

$\exists t$: $e_1  \leq  \ldots \leq e_t > e_{t+1} \geq e_{t+2} \geq \ldots \geq e_n$
&  
$e_i> e_j < e_k$
 & A033321 & 1265  & \ref{section:1265}\\
 
$\exists t$: $e_1  \leq  \ldots \leq e_t > e_{t+1} > e_{t+2} > \ldots > e_n$
&  
$e_i> e_j \leq e_k$
& A071356 & 1064  & \ref{section:1064}\\
 
$\exists t$: $e_1 \leq \ldots \leq e_t \geq e_{t+1} = e_{t+2} = \ldots = e_n$ 
&  
$e_i> e_j \neq e_k$
& See Section 3 & 1079,A  & \ref{section:1079}\\

$\exists t$: $e_1 < \ldots < e_t \geq e_{t+1} = e_{t+2} = \ldots = e_n$ 
&  
$e_i\geq e_j \neq e_k$
& $1+n(n-1)/2$ & 22,C  & \ref{section:22}\\

$\exists t$: $e_1 < \ldots < e_t \geq e_{t+1} \geq e_{t+2} \geq \ldots \geq e_n$ 
&  
$e_i=e_j<e_k$
& $2^{n-1}$  & 64,A  & \ref{section:64}\\

$\exists t$: $e_1 < \ldots < e_t \geq e_{t+1} > e_{t+2} > \ldots > e_n$ 
&  
 $e_i=e_j\leq e_k$
& $F_{n+1}$  & 21  & \ref{section:21}\\

$\exists (t \leq s)$: $e_1 < \ldots < e_t =  \ldots = e_{s} > \ldots > e_n$ 
&  
$e_i \geq e_j\leq e_k$, $e_i \neq e_k$
& $F_{n+2}-1$ & 33,B  & \ref{section:33}  \\

\\

Positive elements monotone:\\
\\
positive entries are strictly decreasing
& 
 $e_i < e_j \leq e_k$
& $2^{n-1}$& 64,B  & \ref{section:64} \\

positive entries are weakly decreasing
& 
$e_i < e_j <  e_k$
& $F_{2n-1}$ & 233   & \ref{section:233}\\

positive entries are strictly increasing
& 
$ e_j \geq  e_k$, $e_i <e_k$
& Catalan number $C_n$ & 429,B  & \ref{section:429} \\

positive entries are weakly increasing
& 
$ e_j >  e_k$, $e_i <e_k$
& large Schr\"oder number  & 1806,A  & \ref{section:1806} \\

\\
Multiplicity constraints:\\
\\

entries $e_2, \ldots, e_n$  are all distinct
& 
$e_i \leq e_j = e_k$
& $2^{n-1}$ & 64,D  & \ref{section:64}\\

$|\{e_1,e_2, \ldots, e_n\}| \leq 2$
& 
$e_i\neq e_j \neq e_k$, $e_i \neq e_k$
& Grassmannian perms & 121,B  & \ref{section:121}\\

positive entries are distinct
& 
$e_i < e_j = e_k$
& Bell numbers & 877,A  & \ref{section:877} \\

no three entries equal
& 
$e_i = e_j = e_k$
& Euler up/down nos. & 1385  & \ref{section:1385}\\

only adjacent entries can be equal
& 
$e_i \neq e_j \neq e_k$, $e_i \neq e_k$
& Bell numbers & 877,C  & \ref{section:877} \\

$e_s=e_t \implies |s-t|\leq 1$
& 
$e_i  = e_k$
& A229046 & 304  & \ref{section:304} \\
&&&&\\

\end{tabular}
}

\hrule

\caption{Characterizations of  inversion sequences avoiding  triples of relations.}
\label{interesting}
\end{table}

\begin{table}
\hrule 
\centering
\resizebox{0.92\textwidth}{!}{

\begin{tabular}{llllll}
Inversion sequences & appear to be & & & & \\
with no $i < j < k$ & counted by & proven?& notes/OEIS  description  & $a_7$, & Sect.\\
 such that: & OEIS seq: &&&equiv class&\\

$e_i \neq e_j$, $e_i \neq e_k$& A004275 & yes & $2(n-1)$ for $n > 1$&12,A & \ref{section:12}\\

$e_i \geq e_j$, $e_i \neq e_k$& A004275 & yes & $2(n-1)$ for $n > 1$&12,B  & \ref{section:12}\\

$e_i = e_j \leq e_k$ & A000045& yes  &Fibonacci numbers, $F_{n+1}$  & 21 & \ref{section:21}\\

$e_i < e_j \not = e_k$ &A000124& yes & Lazy caterer sequence & 22,A & \ref{section:22}\\

$e_i < e_j$, $e_i < e_k$ &A000124& yes & Lazy caterer sequence & 22,B & \ref{section:22}\\

$e_i \geq e_j \not = e_k$&A000124& yes & Lazy caterer sequence & 22,C & \ref{section:22}\\

$e_i \not = e_j \leq e_k$ &A000071& yes & $F_{n+2}-1$  & 33,A & \ref{section:33}\\

$e_i \geq e_j \leq e_k$, $e_i \not = e_k$ &A000071& yes & $F_{n+2}-1$  & 33,B & \ref{section:33}\\

$e_i = e_j < e_k$&A000079 & yes & $\I_n(001)$, $2^{n-1}$ (see \cite{paisI}) & 64,A & \ref{section:64}\\

$e_i < e_j \leq e_k$&A000079 & yes & $2^{n-1}$ & 64,B & \ref{section:64}\\

$e_i < e_j \geq e_k$&A000079 & yes & $2^{n-1}$ & 64,C & \ref{section:64}\\

$e_i \leq  e_j =e_k$&A000079 & yes & $2^{n-1}$ & 64,D & \ref{section:64}\\

$e_i \neq e_j < e_k$  & A000325&yes & Grassmannian permutations & 121,A & \ref{section:121}\\

$e_i \neq e_j \neq e_k$, $e_i \neq e_k$   & A000325& yes& Grassmannian permutations & 121,B & \ref{section:121}\\

$e_j \geq e_k $, $e_i \neq e_k$    & A000325& yes& Grassmannian permutations & 121,C & \ref{section:121}\\

$e_i \not = e_j < e_k$, $e_i \leq e_k$  & A034943& yes &321-avoiding separable perms  & 151 & \ref{section:151}\\

$e_i \neq e_j < e_k$, $e_i \neq e_k$ & A088921& yes &$\Sn_n(321,2143)$  & 185 & \ref{section:185}\\

$e_i \geq e_k$ &A049125 & no  & ordered trees, internal nodes, adj. to $\leq$ 1  leaf & 187 & \ref{section:187} \\

$e_i \leq e_j \geq e_k$, $e_i \not = e_k$& A005183 & yes & $\Sn_n(132,4312)$, $n2^{n-1} +1$ & 193 & \ref{section:193}\\

$e_i < e_j < e_k$ & A001519 & yes & $\I_n(012)$, $F_{2n-1}$  (see \cite{paisI,Mansour}) & 233 & \ref{section:233}\\

$e_i=e_k$ & A229046 & no  & recurrence $\rightarrow$ gf? & 304& \ref{section:304} \\

$e_j > e_k$ & A000108&  yes &Catalan numbers  & 429,A & \ref{section:429}\\

$e_j \geq e_k$, $e_i <e_k$  & A000108& yes &Catalan numbers  & 429,B &\ref{section:429}\\

$e_i \geq e_j $, $e_i \geq e_k$ & A000108& yes  &Catalan numbers \cite{KimLin} &  429,C &\ref{section:429}\\

$e_i \not = e_j = e_k$  & A047970& yes&  $\Sn_n({\bar{3}}  {\bar{1}} 542)$, nexus numbers   & 523  &\ref{section:523}\\

 $e_j \leq e_k$, $e_i \geq e_k$ & A108307& yes  &set partitions avoiding enhanced 3-crossings  & 772,A  &\ref{section:772}\\
 
$e_i \geq e_j \geq e_k$ &A108307& yes &set partitions avoiding enhanced 3-crossings \cite{Lin, Yan}  & 772,B  &\ref{section:772}\\

$e_i < e_j =e_k$ & A000110 & yes &$\I_n(011)$ (see \cite{paisI}), Bell numbers $B_n$ &  877,A  &\ref{section:877}\\

$e_i = e_j \geq e_k$ & A000110 & no &$\I_n(000,110)$, $B_n$ &  877,B &\ref{section:877}\\

$e_j \neq e_k$, $e_i=e_k$  & A000110 &yes & $\I_n(010,101)$, $B_n$ &  877,C &\ref{section:877}\\

$e_i \geq e_j$, $e_i = e_k$  & A000110 & no &$\I_n(000,101)$, $B_n$ &  877,D &\ref{section:877}\\

$e_i > e_j$& A000984 & yes & central binomial coefficients & 924 &\ref{section:924}\\

$e_i > e_j \leq e_k$ & A071356 & no & certain underdiagonal lattice paths & 1064 &\ref{section:1064}\\

$e_i > e_j < e_k$& A033321 &  yes &$\Sn_n(2143,3142,4132)$  (see \cite{BS}) & 1265 &\ref{section:1265}\\

$e_i >e_j$, $e_i \leq e_k$ & A106228 & no & $\I_n(101,102)$, $\Sn_n(4123,4132,4213)$& 1347 &\ref{section:1347}\\

$e_i=e_j=e_k$ & A000111 & yes & $\I_n(000)$  (see \cite{paisI}), Euler up/down numbers& 1385 &\ref{section:1385} \\

$e_i >e_j$, $e_i < e_k$ & A200753 & yes &  $\I_n(102)$, \cite{Mansour} & 1694 &\ref{section:1694}\\

$e_j >e_k$, $e_i <e_k$ & A006318& yes &$\I_n(021)$ \cite{paisI,Mansour}, large Schr\"oder numbers $R_{n-1}$ & 1806,A &\ref{section:1806}\\

$e_i > e_j$, $e_i \geq e_k$ & A006318&yes  &$\I_n(210,201,101,100)$, $R_{n-1}$ & 1806,B &\ref{section:1806}\\

$e_i \geq e_j$, $e_i >e_k$  & A006318&  yes&$\I_n(210,201,100,110)$, $R_{n-1}$  & 1806,C &\ref{section:1806}\\

$e_i \geq e_j \not = e_k$, $e_i \geq e_k$ & A006318& yes& $\I_n(210,201,101,110)$, $R_{n-1}$  & 1806,D &\ref{section:1806}\\

$e_i \geq e_j \geq e_k$, $e_i > e_k$ & A001181&yes  &Baxter permutations \cite{KimLin}   & 2074 &\ref{section:2074}\\

$e_i > e_j$, $e_i > e_k$ & A098746 & no &$\I_n(210,201,100)$,  $\Sn_n(4231,42513)$ & 2549,A &\ref{section:2549}\\

$e_i > e_j \neq e_k$, $e_i \geq e_k$ & A098746 & no & $\I_n(210,201,101)$, $\Sn_n(4231,42513)$ & 2549,B &\ref{section:2549}\\

$e_i \geq e_j \neq e_k$, $e_i > e_k$ & A098746 & no &  $\I_n(210,201,110)$, $\Sn_n(4231,42513)$ & 2549,C &\ref{section:2549}\\

$e_j < e_k$, $e_i\geq e_k$& A117106  &yes  &$\I_n(201,101)$, $\Sn_n(21{\bar{3}}54)$ & 2958,A &\ref{section:2958}\\

$e_i > e_j \geq e_k$  & A117106  & yes &  $\I_n(210,100)$,  $\Sn_n(21{\bar{3}}54)$
\cite{BouvelGuerriniEtAl} & 2958,B &\ref{section:2958}\\

$e_i \geq e_j > e_k$& A117106  &yes  &  $\I_n(210,110)$,  $\Sn_n(21{\bar{3}}54)$ & 2958,C &\ref{section:2958}\\

$e_j \leq e_k$, $e_i > e_k$& A117106  &yes  &  $\I_n(201,100)$, $\Sn_n(21{\bar{3}}54)$ & 2958,D &\ref{section:2958}\\

$e_j< e_k$, $e_i = e_k$&A113227 & yes & $\I_n(101)$, $\Sn_n($1-23-4$)$, (see \cite{paisI})& 3207,A &\ref{section:3207}\\

$e_i=e_j>e_k$&A113227 & yes & $\I_n(110)$, $\Sn_n($1-23-4$)$, (see \cite{paisI})& 3207,B &\ref{section:3207}\\

$e_i > e_j \not = e_k$, $e_i > e_k$ &  A212198 & yes &$\I_n(201,210)$, $MMP(0,2,0,2)$-avoiding perms  & 3720 &\ref{section:3207}\\
\end{tabular}
}
\vspace{.1in}
\hrule
\vspace{.1in}
\caption{Patterns whose avoidance sequences  appear to match sequences in the OEIS.  Those marked as ``yes'' are cited, if known, and otherwise are proven in this paper.} 
\label{other triples}
\end{table}

\vspace{.1in}

\begin{table}
\hrule 
\centering
\resizebox{1\textwidth}{!}{
\begin{tabular}{llll}

Inversion sequences&&&  \\
with no $i < j < k$&  comments & initial terms $a_1, \ldots a_9$ &  $a_7$, \\
such that:&&& equiv class \\
\\
$e_j \geq e_k$, $e_i \geq e_k$& $\I_n(000,010,011,021)$ (A279544)&  $1,2,4,10,26,73,214,651,2040$  & 214\\
$e_i \leq e_j$, $e_i \geq e_k$& $\I_n(000,010,110,120)$ (A279551)& $1,2,4,10,27,79,247,816,2822$  & 247\\
$e_j \geq e_k$, $e_i = e_k$& $\I_n(000,010)$ (A279552)& $1,2,4,10,29,95,345,1376,5966$&345\\
$e_j \neq e_k$, $e_i \geq e_k$ & Wilf-eq. to 663B (Sec. \ref{equivalences}, A279553)& $1,2,5,15,50,178,663,2552,10071$&663,A\\
$e_i \neq e_j$, $e_i \geq e_k$&Wilf-eq. to 663A (Sec. \ref{equivalences}, A279553)&$1,2,5,15,50,178,663,2552,10071$ &663,B\\
$e_i \neq e_j \neq e_k$, $e_i \geq e_k$& $\I_n(010,101,120,201)$ (A279554)& $1,2,5,15,51,188,733,2979,12495$&733\\
$e_j >e_k$, $e_i \geq e_k$& Wilf-eq. to 746B (Sec. \ref{equivalences}, A279555) &$1,2,5,15,51,189,746,3091,13311$ &746,A\\
$e_i \neq e_j\geq e_k $, $e_i \geq e_k$&Wilf-eq. to 746A (Sec. \ref{equivalences}, A279555)  & $1,2,5,15,51,189,746,3091,13311$&746,B\\
$e_i \leq e_j\neq e_k$, $e_i \geq e_k$ & $\I_n(010,110,120)$ (A279556)&$1,2,5,15,51,190,759,3206,14180$ &759\\
$e_i \leq e_j > e_k$, $e_i \not= e_k$ & counted - See Section \ref{counting} (A279557) & $1,2,6,20,68,233,805,2807,9879$ & 805\\
$e_i \neq e_j>e_k$, $e_i \geq e_k$&$\I_n(010,120,210)$ (A279558)&$1,2,5,15,52,200,830,3654,16869$ &830\\
$e_i < e_j $, $e_i \geq e_k$ & $\I_n(010,120)$  (A279559)& $1,2,5,15,52,201,845,3801,18089$ & 845\\
$e_j >e_k$, $e_i =e_k$&  $\I_n(010)$ (A263779)& $1,2,5,15,53,215,979,4922,26992$&979\\
$e_i >e_j$, $e_i \neq e_k$& counted - See Section \ref{counting}  (A279560)& $1,2,6,21,76,277,1016,3756, 13998$&1016\\
$e_i > e_j \not = e_k$ &  counted - See Section \ref{counting}   (A279561)&$1,2,6,21,77,287,1079,4082,15522$ & 1079,A\\
$e_i <e _j >e_k$, $e_i \neq e_k$  &  $\I_n(021,120)$ (A279561) &$1,2,6,21,77,287,1079,4082,15522$&1079,B\\
$e_i > e_j \leq e_k$, $e_i \not = e_k$ & $\I_n(100,102,201)$ (A279562)&$1,2,6,21,78,299,1176,4729,19378$ & 1176\\
$e_i > e_j \not = e_k$, $e_i \not = e_k$ & $\I_n(102, 201, 210)$ (A279563)& $1,2,6,22,85,328,1253,4754, 17994$ & 1253\\
$e_i \geq e_j = e_k$ & $\I_n(000,100)$ (A279564)& $1,2,5,16,60,260,1267,6850,40572$ & 1267\\
$e_i >e_k$ &$\I_n(100,110,120,210,201)$ (A279565)&$1,2,6,21,81,332,1420,6266,28318$ &1420\\
$e_i > e_j < e_k$, $e_i \not = e_k$ & $\I_n(102,201)$ (A279566)& $1,2,6,22,87,354,1465,6154,26223$ & 1465\\
$e_j \geq e_k$, $e_i>e_k$&$\I_n(100,110,120,210)$ (A279567)& $1,2,6,21,82,343,1509,6893,32419$&1509\\
$e_j \neq e_k$, $e_i >e_k$&Wilf-eq. to 1833B (Sec. \ref{equivalences}, A279568) &$1,2,6,22,90,396,1833,8801,43441$ &1833,A\\
$e_i \neq e_j$ , $e_i > e_k$&Wilf-eq. to 1833A (Sec. \ref{equivalences}, A279568) &$1,2,6,22,90,396,1833,8801,43441$ &1833,B\\
$e_j >e_k$, $e_i > e_k$ & Wilf-eq. to 1953B (Sec. \ref{equivalences}, A279569)&$1,2,6,22,91,409,1953,9763,50583$ &1953,A\\
$e_i \neq e_j \geq e_k$, $e_i >e_k$&Wilf-eq. to 1953A (Sec. \ref{equivalences}, A279569) & $1,2,6,22,91,409,1953,9763,50583$&1953,B\\
$e_i \leq e_j > e_k$, $e_i > e_k$ & $\I_n(110,120)$ (A279570)& $1,2,6,22,92,423,2091,10950,60120$ & 2091\\
$e_i > e_j \leq e_k$, $e_i \geq e_k$& $\I_n(100,101,201)$ (A279571)&$1,2,6,22,92,424,2106,11102,61436$ &2106\\
$e_i \neq e_j \neq e_k$, $e_i > e_k$&$\I_n(120,210,201)$ (A279572)& $1,2,6,23,101,484,2468,13166,72630$&2468\\
$e_i \not = e_j > e_k$, $e_i > e_k$ & $\I_n(210,120)$ (A279573)& $1,2,6,23,102,499,2625,14601,84847$ & 2625\\
$e_i < e_j$, $e_i > e_k$&$\I_n(120)$ (A263778) &$1,2,6,23,103,515,2803,16334,100700$ &2803\\
$e_i > e_j = e_k$&$\I_n(100)$  (A263780) & $1,2,6,23,106,565,3399,22678,165646$&3399\\
$e_j < e_k$, $e_i > e_k$& $\I_n(201)$ (A263777)& $1,2,6,24,118,674,4306,29990,223668$&4306,A\\
$e_i > e_j > e_k$ & $\I_n(210)$  Wilf-eq to 4306A \cite{paisI}& $1,2,6,24,118,674,4306,29990,223668$&4306,B\\

\end{tabular}
}
\vspace{.1in}
\hrule
\vspace{.1in}

\caption{The patterns whose avoidance sequences did not match sequences in the OEIS.  (OEIS numbers in parentheses were newly assigned.)}
\label{questions}
\end{table}

\section{Patterns whose sequences appear in the OEIS}

\subsection{Classes 7(A,B,C,D): $n$}
\label{section:7}

There are four equivalence classes of patterns whose avoidance sequences are counted by the positive integers.  We characterize each, from which it is straightforward to prove that
$$
|\I_n(-,<,-)|=
|\I_n(-,\geq,-)|=
|\I_n(=,-,-)|=
|\I_n(\neq,-,-)|=n,
$$
although these four patterns are not equivalent. The conditions on the entries of inversion sequences in each class are as follows.

\begin{itemize}
\item (Class {\bf 7A}:   $e_j < e_k$) $\I_n(-,<,-)$ is the set of $e \in \I_n$ satisfying $e_2 \geq e_3 \geq \ldots \geq e_n$.
\item (Class {\bf 7B}: $e_j \geq e_k$) $\I_n(-,\geq,-)$ is the set of $e \in \I_n$ satisfying $e_1 \leq e_2  < e_3 <  \ldots < e_n$.
\item (Class {\bf 7C}: $e_i = e_j$) $\I_n(=,-,-)$ is the set of $e \in \I_n$ satisfying $e_1 < e_2   <  \ldots < e_{n-1}$.
\item (Class {\bf 7D}: $e_i \neq e_j$) $\I_n(\neq,-,-)$ is the set of $e \in \I_n$ satisfying $e_1=e_2   =  \ldots =e_{n-1}$.
\end{itemize}

{\bf Note:} Simion and Schmidt \cite{SimionSchmidt} showed that $|\Sn_n(123,132,231)|=n$. In fact, we get the following relationship between inversion sequences and permutations.

\begin{theorem}
For every $n$, $\Theta(\Sn_n(123,132,231))=\I_n(=,-,-)$.
\end{theorem}

\subsection{Classes 12(A,B):  $2(n-1)$ for $n>1$}
\label{section:12}

In the following theorem, we show that classes {\bf 12A} ($e_i \neq e_j$ and $e_i \neq e_k$) and {\bf 12B} ($e_i \geq e_j$ and $e_i \neq e_k$) are Wilf-equivalent, but not equivalent. 
\begin{theorem}
$|\I_n(\neq, -, \neq)|$ and $|\I_n(\geq, -, \neq)|$ are both counted by $1$ if $n=1$ and by $2(n-1)$ for $n>1$.
However, $\I_n(\neq, -, \neq) \neq \I_n(\geq, -, \neq)$ for $n > 2$.
\end{theorem}
\begin{proof}

For $n=1$ this is clear.   For $n > 1$
this follows by noting that any $e\in \I_n$ with no $i < j < k$ such that $e_i \neq e_j$ and $e_i \neq e_k$ can be $(0,0,\ldots,0)$ or can be of the form $(0,0,\ldots, t, 0)$ or $(0,0,\ldots, 0,s)$ where $t \in [n-2]$ and $s \in [n-1]$.
On the other hand,
any $e\in \I_n$ with no $i < j < k$ such that $e_i \geq e_j$ and $e_i \neq e_k$ must
have the form $(0,1,2, \ldots, n-2,t)$ for $t=0, \ldots, n-1$ or the form
$(0,1,2, \ldots, t-1,t,t, \ldots, t)$ for $t=0, \ldots, n-3$.
\end{proof}

\subsection{Class 21: $F_{n+1}$}
\label{section:21}

Let $F_n$ be the $n$-th Fibonacci number, where $F_0=0$, $F_1=1$, and $F_n=F_{n-1}+F_{n-2}$ for $n \geq 2$.  The Fibonacci numbers count pattern-avoiding permutations such as  $\Sn_n(123,132,213)$ \cite{SimionSchmidt}. Class {\bf 21} ($e_i = e_j \leq e_k$) is counted by the $(n+1)$-th Fibonacci number, as shown below.

\begin{observation}
The inversion sequences $e \in \I_n$  with no $i<j<k$ such that  $e_i   = e_j  \le e_k$ are those satisfying, 
for some $t \in [n]$,
\begin{equation}
 e_1<e_2< \ldots < e_{t} \geq e_{t+1}> \ldots > e_n.
\label{equation:21}
\end{equation}
\end{observation}

\begin{theorem}
$|\I_n(=,\leq, -)|=F_{n+1}$
\end{theorem}
\begin{proof}
This is clear for $n=1,2$. For $n \geq 3$, any $e \in \I_n(=, \leq, -)$ must have the form
$(0, e_1+1, \ldots, e_{n-1}+1)$ for $(e_1, \ldots, e_{n-1}) \in \I_{n-1}(=, \leq, -)$ or
$(0, e_1+1, \ldots, e_{n-2}+1,0)$ for $(e_1, \ldots, e_{n-2}) \in \I_{n-2}(=, \leq, -)$.  Conversely, strings of either of these forms
are in $\I_n(=, \leq, -)$.
\end{proof}

Among the 343 patterns checked, it can be shown that the six patterns whose avoidance sequence is counted by $F_{n+1}$  are equivalent.
\begin{observation}
All of the following patterns are equivalent to $(=, \leq, -)$:
$(=,-,\leq)$,
$(=,\leq,\leq)$,
$(\geq,-,\leq)$,
$(\geq,\leq,-)$,
$(\geq,\leq,\geq)$.
\end{observation}

\subsection{Classes 22(A,B,C):  Lazy caterer sequence, ${n \choose 2} +1$}
\label{section:22}

We show that there are three inequivalent patterns that are all counted by the sequence ${n \choose 2} +1$, which also counts $\Sn_n(132,321)$ \cite{SimionSchmidt}.

\subsubsection{Class {\bf 22A}: Avoiding $e_i < e_j \neq e_k$}

It is not hard to see that inversion sequences avoiding $e_i < e_j  \ne e_k$ are characterized by the following.
\begin{observation} The inversion sequences $e \in \I_n$  with no $i<j<k$ such that  $e_i < e_j  \ne e_k$ are those satisfying, 
for some $t$ where $1 \leq t \leq n$,
\begin{equation}
0 = e_1=e_2= \ldots = e_{t-1} \leq e_t = e_{t+1} = \ldots = e_n.
\label{equation:22A}
\end{equation}
\end{observation}
That is, either $e=(0,0,\ldots,0)$ or, for some $t: 2 \leq t \leq n$ and $j: 1 \leq j \leq t-1$, $e$ consists of a string of $t-1$ zeros followed by a string of $n-t+1$ copies of $j$.  This gives the following.

\begin{theorem}
$|\I_n(<,\neq,-)|= \binom{n}{2} + 1$.
\end{theorem}

The sequence whose $n$th entry is $\binom{n}{2} + 1$ is sequence A000124 in the OEIS, where it is called the {\em Lazy Caterer} sequence \cite{Sloane} because its $n$th entry is  the maximum number of pieces that can be formed by making $n-1$ straight cuts in a pizza.  This is also the avoidance sequence for certain pairs of permutation patterns, as was shown by Simion and Schmidt \cite{SimionSchmidt}.

\begin{theorem}[Simion-Schmidt \cite{SimionSchmidt}]
$|\Sn_n(\alpha, \beta)|=\binom{n}{2} + 1$
for any of the following pairs  $(\alpha, \beta)$ of patterns:
\[
(132,321), \ (123,231), \ (123,312), \ (213,321).
\]
\label{SimionSchmidtCaterer}
\end{theorem}

We can relate these permutations to the inversion sequences in $\I_n(<,\neq,-)$.  Recall the bijection  $\Theta(\pi):  \Sn_n \rightarrow \I_n$ for  $\pi=\pi_1 \ldots \pi_n \in \Sn_n$  defined by $\Theta(\pi)=(e_1,e_2, \ldots, e_n)$, where $e_i = |\{j \ | \ j < i \ {\rm and} \ e_j > e_i \}|$.  

\begin{theorem}
$ \I_n(<,\neq,-)= \Theta(\Sn_n(213,321))$.
\end{theorem}
\begin{proof}
Note that $e \in \I_n$ satisfies \eqref{equation:22A}
if and only if
 $\pi = \Theta^{-1}(e)$ satisfies
\[
\pi_1 < \pi_2 < \ldots < \pi_t  > \pi_{t+1}  < \pi_{t+2} < \ldots < \pi_n,
\]
where $\pi_t, \pi_{t+1}, \ldots, \pi_n$ are consecutive integers.
Such permutations are precisely the ones that avoid both $213$ and $321$.
\end{proof}

The patterns $( <, -, <)$ and $(\geq, \not =, -)$ are
Wilf-equivalent to the pattern $(< ,\ne , -)$  on inversion sequences, although the three patterns are pairwise inequivalent.  This is clear from the following characterizations.

\subsubsection{Class {\bf 22B}: Avoiding $e_i < e_j$ and $e_i < e_k$}

\begin{observation} The inversion sequences $e \in \I_n$  with no $i<j<k$ such that  $e_i < e_j$ and $e_i < e_k$  are those satisfying, 
for some $t$ where $1 \leq t \leq n$,
\begin{equation}
0 = e_1=e_2= \ldots = e_{t-1} \leq e_t \geq e_{t+1} = \ldots = e_n =0.
\label{equation:22B}
\end{equation}
\end{observation}

\subsubsection{Class {\bf 22C}: Avoiding $e_i \geq e_j \neq e_k$}

\begin{observation} The inversion sequences $e \in \I_n$  with no $i<j<k$ such that  $e_i \geq e_j  \ne e_k$ are those satisfying, 
for some $t$ where $1 \leq t \leq n$,
\begin{equation}
e_1 < e_2 < \ldots < e_{t-1} \geq e_t = e_{t+1} = \ldots = e_n.
\label{equation:22C}
\end{equation}
\end{observation}

\subsection{Classes 33(A,B):  $F_{n+2}-1$}
\label{section:33}

We show that {\bf 33A}: $(\neq, \leq , -)$ and {\bf 33B}:  $(\geq, \leq, \neq)$ are inequivalent Wilf-equivalent patterns whose avoidance sequences are counted by  $F_{n+2}-1$.

\subsubsection{Class {\bf 33A}: Avoiding $e_i  \not = e_j  \le e_k$}

\begin{theorem}
$|\I_n(\ne , \leq, -)|= F_{n+2}-1$.
\end{theorem}
\begin{proof}
Observe that 
the inversion sequences $e \in \I_n$  with no $i<j<k$ such that  $e_i  \not = e_j  \le e_k$ are those satisfying, 
for some $t$ where $1 \leq t \leq n+1$,
\begin{equation}
0 = e_1=e_2= \ldots = e_{t-1} < e_t > e_{t+1}> \ldots > e_n.
\label{equation:33A}
\end{equation}

We can partition the inversion sequences in $\I_n(\neq, \leq, -)$ into three disjoint sets: $\{(0,0,\ldots,\allowbreak0)\}$, $A = \{e \in\I_n(\neq, \leq, -) \mid e_n \neq 0\}$, and $B = \{e \in\I_n(\neq, \leq, -) \mid e \neq 0, e_n = 0\}$.  Any inversion sequence in $A$ can be constructed by taking any $e' = (e_1,e_2,\ldots,e_{n-1}) \in \I_{n-1}(\neq,\leq,-)$ and letting $t$ be the index of the first nonzero entry (if there is no nonzero entry, set $e_t=e_{n-1}$).  Then we can use the characterization given by  (\ref{equation:33A}) to verify that $(0,e_1,e_2,\ldots,e_{t-1},e_t+1,\ldots,e_{n-1}+1)$ is an element of $A$.  

Any element of $B$ can be constructed by taking some $e'' = (e_1,e_2,\ldots,e_{n-2}) \in \I_{n-2}(\neq,\leq,-)$ and letting $t$ be the index of the first nonzero entry (again, if no such entry exists, set $e_t=e_{n-2}$).  Then $(0,e_1,\ldots,e_{t-1},e_t+1,\ldots,e_{n-2}+1,0)$ is an element of $B$.  

Setting $a_n=|\I_n(\ne , \leq, -)|$, this gives $a_n=a_{n-1}+a_{n-2} + 1$, with initial conditions $a_1=1$, $a_2=2$. So $a_n = F_{n+2}-1$.
\end{proof}

\subsubsection{Class {\bf 33B}:  Avoiding $e_i  \geq  e_j  \leq e_k$ and $e_i \neq e_k$}

\begin{theorem}
$|\I_n(\ge,  \leq, \ne)|= F_{n+2}-1$.
\end{theorem}
\begin{proof}
The inversion sequences $e \in \I_n$  with no $i<j<k$ such that  $e_i  \geq  e_j  \le e_k$ and $e_i \ne e_k$ are those satisfying, 
for some $t,s$ where $1 \leq t  \leq s \leq n$,
\begin{equation}
e_1<e_2< \ldots < e_{t-1} <e_t = e_{t+1}= \ldots = e_s> e_{s+1} > \ldots > e_n.
\label{equation:21B}
\end{equation}

The following is a bijection mapping $\I_n(\ge,  \leq, \ne)$ to $\I_n(\neq,\leq,-)$.
For $e \in \I_n(\ge,  \leq, \ne)$, let $s$ be the first index, if any, such that $e_s > e_{s+1}$; if $e$ is weakly increasing, set $s=n$. To obtain an element of $ \I_n(\neq,\leq,-)$, set $e_i=0$ for $i=1, \ldots, s-1$.
\end{proof}

\subsection{Classes 64(A,B,C,D):   $2^{n-1}$}
\label{section:64}

\subsubsection{Class 64A:  Avoiding $e_i  =  e_j  < e_k$}

Corteel, et al.\ \cite{paisI} characterized $\I_n(=,<,-)=\I_n(001)$ as the set of $e \in \I_n$ satisfying, for some $t \in [n]$,
\[
e_1 < e_2 < \ldots < e_t \geq e_{t+1} \geq e_{t+2} \geq \ldots \geq e_n.
\]
They showed that $|\I_n(001)|=2^{n-1}$ by showing that the bijection $\Theta: \Sn_n \rightarrow \I_n$ restricts to a bijection from  $\Sn_n(132,231)$ to $\I_n(001)$.  Simion and Schmidt \cite{SimionSchmidt} showed that permutations avoiding both 132 and 231 are enumerated by $2^{n-1}$.  

We show that three other patterns are Wilf-equivalent, though inequivalent, to class {\bf 64A}.

\subsubsection{Class 64B:  Avoiding $e_i<e_j\le e_k$}

\begin{theorem}
The number of $e \in \I_n$ with no  $i < j < k$ such that  $e_i<e_j\le e_k$  is $2^{n-1}$.
\end{theorem}
\begin{proof}
First observe that the inversion sequences  $e$ with no  $i < j < k$ such that  $e_i<e_j\le e_k$   are those
whose positive entries form a strictly decreasing sequence.

Let $B_n = \I_n(<, \leq,-)$.  Notice that $|B_1| = |\{(0)\}|=1$; we will show that 
for $n > 1$, $|B_n| = 2|B_{n-1}|$.
Recall that $\sigma_1(e)$ adds 1 to each positive element in $e$.  


An $e \in B_n$ has no ``$1$'' if and only if it has the form $0 \cdot \sigma_1(e')$ for some $e' \in B_{n-1}$, so there are $|B_{n-1}|$ such $e$.
An $e \in B_n$ has a ``$1$'' if and only if it has the form $e' \cdot 0$ for some $e' \in B_{n-1}$ containing a  ``$1$'' or the form $e' \cdot 1$ for some $e' \in B_{n-1}$  not containing a  ``$1$'', so there are also $|B_{n-1}|$  elements of $B_n$ containing a ``$1$''.
\end{proof}

\subsubsection{Class {\bf 64C}:  Avoiding $e_i<e_j \geq e_k$}

\begin{theorem}
The number of $e \in \I_n$ with no $i < j < k$ such that $e_i<e_j\ge e_k$ is $2^{n-1}$.
\label{theorem:64C}
\end{theorem}
\begin{proof}
The inversion sequences avoiding the pattern $e_i<e_j\ge e_k$, where $i < j < k$, are those $e \in \I_n$
satisfying, for some $t \in [n]$,
\[
0 = e_1 =e_2 = \cdots = e_t < e_{t+1} < e_{t+2} < \cdots < e_n.
\]

Map $e \in \I_n(<,\geq,-)$ to the set consisting of its nonzero values.  Clearly this is a bijection from 
$ \I_n(<,\geq,-)$ to $2^{[n-1]}$.
\end{proof}

In fact, we can show that $\I_n(<,\geq,-)$ is  the image under $\Theta$ of $\Sn_n(213,312)$.

\begin{theorem}
$\Theta(\Sn_n(213,312))=\I_n(<,\geq,-)$.
\end{theorem}

\begin{proof}
It is straightforward to prove that $\Sn_n(213,312)$ consists of the unimodal permutations where \[\pi_1<\pi_2<\cdots<\pi_t=n>\pi_{t+1}>\cdots>\pi_n.\] The inversion sequences avoiding the pattern $e_i<e_j\ge e_k$, where $i < j < k$, are those $e \in \I_n$
satisfying, for some $t \in [n]$,
\[
0 = e_1 =e_2 = \cdots = e_t < e_{t+1} < e_{t+2} < \cdots < e_n.
\]
It immediately follows that $\Theta(\Sn_n(213,312))=\I_n(<,\geq,-)$.
\end{proof}

\subsubsection{Class {\bf 64D}: Avoiding $e_i \leq e_j = e_k$}  

\begin{theorem}
The number of $e \in \I_n$ with no $i < j < k$ such that $e_i\leq e_j = e_k$ is $2^{n-1}$.
\label{theorem:64D}
\end{theorem}
\begin{proof}
The inversion sequences avoiding the pattern $e_i\leq e_j=e_k$, where $i < j < k$, are those $e \in \I_n$
in which all of the entries $e_2, e_3, \ldots, e_n$ are distinct.

Let $D_n = \I_n(\leq,=,-)$.
Then $|D_1| = |\{(0)\}|=1$.  We show that 
for $n > 1$, $|D_n| = 2|D_{n-1}|$.
An $e \in D_n$ ends in $n-1$ if and only if $e=e' \cdot (n-1)$ for some $e' \in D_{n-1}$, so there are $|D_{n-1}|$ such $e$.
An $e \in D_n$ ends in $d \not = n-1$ if and only if $e=e' \cdot d$ where $e' \in D_{n-1}$ and $d$ is the unique element in  $\{0,1, \ldots, n-2\} \setminus \{e_2, \ldots, e_{n-1}\}$, so there are again  $|D_{n-1}|$ such $e$.
\end{proof}

\subsection{Classes 121(A,B,C):   Grassmannian permutations,  $2^n-n$}
\label{section:121}

Permutations with at most one descent were called {\em Grassmannian} by Lascoux and Sch{\"u}tzen-berger \cite{LS},  who also characterized them in terms of their Lehmer codes.  Grassmannian permutations of length $n$ are counted by  $2^n-n$ and relate to three equivalence classes of patterns for inversion sequences.

\subsubsection{Class {\bf 121A}: Avoiding $e_i\neq e_j < e_k$}

\begin{theorem}
$|\I_n(\neq, < , -)|=2^n-n$.
\label{thm:121A}
\end{theorem}
\begin{proof}
First observe that those $e \in \I_n$ with no $i < j < k$ such that $e_i\neq e_j < e_k$ are exactly those with at most one ascent.

Using the mapping $\Theta:  \Sn_n \rightarrow \I_n$, recall that
$\pi$ has a descent in a position $i$ if and only if $\Theta(\pi)$ has an ascent in position $i$.
Thus $\Theta$ restricts to a bijection from Grassmannian permutations of $[n]$ to $\I_n(\neq, < , -)$.
\end{proof}

The inversion sequences in $\I_n(\neq, < , -)$ correspond to the Grassmannian Lehmer codes used by Lascoux and Sch{\"u}tzenberger \cite{LS} via the natural bijection (reversal) between inversion sequences and Lehmer codes.

\subsubsection{Class {\bf 121B}: Avoiding $e_i\neq e_j \neq  e_k$ and $e_i \neq e_k$}

\begin{theorem}
$|\I_n(\neq,\neq,\neq)|=2^{n}-n$.
\label{thm:121B}
\end{theorem}
\begin{proof}
Note that inversion sequences with no $i < j < k$ such that $e_i\neq e_j \neq  e_k$ and $e_i \neq e_k$ are those with
at most 2 distinct entries; precisely, $|\{e_1, \ldots, e_{n}\}| \leq 2$.

The theorem is clear for $n=1$.  Now consider some $e \in \I_n(\neq,\neq,\neq)$ when $n>1$.  Note that $(e_1, \ldots, e_{n-1}) \in \I_{n-1}(\neq,\neq,\neq)$.  It follows that either (1)
$|\{e_1, \ldots, e_{n-1}\}|=2$, and $e_n$ is one of the two elements occurring in $(e_1, \ldots, e_{n-1})$; or (2)
$|\{e_1, \ldots, e_{n-1}\}|=1$, and $e_n \in \{0, 1, \ldots, n-1\}$. Furthermore, the only inversion sequence in $\I_{n-1}(\neq,\neq,\neq)$ where $|\{e_1, \ldots, e_{n-1}\}|=1$ is the zero inversion sequence.  This gives the recurrence
$|\I_n(\neq,\neq,\neq)|=2( |\I_{n-1}(\neq,\neq,\neq)| -1) + n$ which has the claimed solution.
\end{proof}

\subsubsection{Class {\bf 121C}: Avoiding $e_j \geq e_k $ and $e_i \neq e_k$  }

\begin{theorem}
$|\I_n(-,\geq,\neq)|=2^{n}-n$.
\end{theorem}
\begin{proof}
Inversion sequences with no $i < j < k$ such that $e_j\geq e_k$ and $e_i \neq  e_k$  are those 
satisfying
\[
e_1= \ldots = e_{i-1} < e_i < \ldots < e_n
\]
or, if $e_{i+1}=0$, 
\[
e_1= \ldots = e_{i-1} < e_i   < e_{i+2} < \ldots < e_n,
\]
for some $i$ with $2 \leq i \leq n+1$.

To count these, for each  $t=1, \ldots, n-1$, and for any $t$-element subset $x_1<x_2< \ldots < x_t$ of $[n-1]$, associate the length~$n$ inversion sequence $(0,0,\ldots,0, x_1,x_2, \ldots, x_t)$ and, unless $\{x_1, \ldots, x_t\}=
\{n-t, n-t+1, \ldots, n-1\}$, also associate the length~$n$ inversion sequence $(0,0,\ldots,0, x_1,0,x_2, \ldots, x_t)$,
giving $2^{n-1} + (2^{n-1} - n)=2^n-n$.
\end{proof}


\subsection{Class 151:  $321$-avoiding separable permutations}
\label{section:151}

We show that the avoidance sequence for this pattern satisfies the recurrence
$a_n=3a_{n-1}-2a_{n-2} + a_{n-3}$ with initial conditions $a_1=1$, $a_2=2$, and $a_3=5$.
This coincides with sequence A034943 in the OEIS, where, among other things, it is said to count $321$-avoiding separable permutations (OEIS entry by Vince Vatter) \cite{Sloane}.  A {\em  separable}  permutation is one that avoids 2413 and 3142.  Moreover, we show that $(\neq,<,\leq)$-avoiding inversion sequences have a simple characterization.

\begin{theorem}
Let $A_n = \I_n(\neq,<,\leq)$ and $a_n=|A_n|$. Then $a_n=3a_{n-1}-2a_{n-2} + a_{n-3}$ with initial conditions $a_1=1$, $a_2=2$, and $a_3=5$.
\end{theorem}
\begin{proof}
First, it can be shown
that the set of $e \in \I_n$ such that there is no $i < j < k$ for which $e_i \neq e_j <  e_k$ and $e_i \leq e_k$ is the set of
$e \in \I_n$  where the nonzero elements are weakly decreasing and equal nonzero elements are consecutive.
That is, (1) if $e_i < e_j$, then $e_i=0$ and (2) if $0 < e_i = e_j$ for some $i < j$, then
$e_i = e_{i+1} = \ldots = e_j$.

Define $X_n,Y_n,Z_n$ by
\begin{eqnarray*}
X_n &=& \{ e \in A_n \ | \  e_i \neq 1, \ {\rm for \ all} \ 1 \leq i \leq n\},\\
Y_n &=& \{ e \in A_n \ | \  e_n=1 \},\\
Z_n &=& \{ e \in A_n \ | \  e_n =0 \ {\rm and} \ e_i=1 \ {\rm for \ some} \  i<n\}.\\
\end{eqnarray*}
Then $A_n$ is the disjoint union  $A_n = X_n \cup Y_n \cup Z_n$. 
Recall that the operator $\sigma_1$ adds 1 to the positive elements of an inversion sequence. To get a recurrence, note that
$|X_n|=|A_{n-1}|=a_{n-1}$ since $e \in A_{n-1}$ if and only if $0\cdot \sigma_1(e) \in X_n$.
Also, $(e_1, \ldots, e_{n-1}, 0)  \in Z_n$ if and only if $(e_1, \ldots, e_{n-1}) \in Y_{n-1} \cup Z_{n-1}=A_{n-1}-X_{n-1}$; so
$|Z_n|=|A_{n-1}|-|X_{n-1}|=a_{n-1}-a_{n-2}$.
Finally, $(e_1, \ldots, e_{n-1}, 1)  \in Y_n$ if and only if $(e_1, \ldots, e_{n-1}) \in A_{n-1}-Z_{n-1}$, so
$|Y_n|=a_{n-1}-|Z_{n-1}|=a_{n-1}-(a_{n-2}-a_{n-3})$.
Putting this together,
\[
a_n = |A_n| = |X_n|+  |Y_n| + |Z_n| = 3a_{n-1}-2a_{n-1}+a_{n-3}
\]
and the result follows by checking the initial conditions.
\end{proof}

\subsection{Class 185: 321-avoiding vexillary permutations,  $2^{n+1}-{n+1 \choose 3}-2n-1$ }
\label{section:185}

Vexillary permutations, studied by Lascoux and Sch{\"u}tzenberger \cite{LS}, are 2143-avoiding permutations.  
The 321-avoiding vexillary permutations arose in  work of Billey, Jockush and Stanley \cite{BJS} on the combinatorics of Schubert polynomials.  It was shown that $|\Sn_n(321,2143)|\allowbreak= 2^{n+1}-{n+1 \choose 3}-2n-1$ which is
entry A088921 in the OEIS.  In this entry, it is noted that the $321$-avoiding vexillary permutations are exactly the Grassmannian permutations (see Section \ref{section:121}) and their inverses.

We show  that the $(\neq,<,\neq)$-avoiding inversion sequences are counted by the same function as the 
321-avoiding vexillary permutations.

\begin{lemma}
$\I_n(\neq, <, \neq)= \I_n(\neq, \neq, \neq) \cup \I_n(\neq, <, -)$.
\begin{proof}
If $e \in \I_n(\neq, <, \neq)$, then either $e \in \I_n(\neq, <, -)$ or for any $i < j < k$ such that $e_i \neq e_j < e_k$, $e_i=e_k$ and therefore $e \in \I_n(\neq, \neq, \neq)$.

Conversely, if, for some $i < j < k$, $e_i \neq e_j < e_k$ and $e_i \neq e_k$, then $e$ contains both
$(\neq,<,-)$ and $(\neq,\neq,\neq)$.
\end{proof}
\end{lemma}

\begin{theorem}
$ |\I_n(\neq,<,\neq)|=2^{n+1}-{n+1 \choose 3}-2n-1$.
\end{theorem}
\begin{proof}
By Theorem \ref{thm:121B}, $|\I_n(\neq, \neq, \neq)|= 2^n-n$ and by Theorem \ref{thm:121A}, $|\I_n(\neq, <, -)|= 2^n-n$.
From the characterizations of these sets in the proof of Theorems \ref{thm:121A} and \ref{thm:121B},
$\I_n(\neq, \neq, \neq) \cap \I_n(\neq, <, -)$ is the set of inversion sequences with at most one ascent and at most two distinct elements,
that is, the set of $e \in \I_n$ satisfying, for some $1 \leq t < a < b \leq n+1$,

\[
0 = e_1 = \ldots = e_{a-1}; \ \ \ \  t=e_a = \ldots = e_{b-1}; \  \ \ \ 0 = e_b = \ldots = e_n,
\]
which is counted by ${n+1 \choose 3}$, together with $(0,0, \ldots,0)$.
Thus
\[
|\I_n(\neq, \neq, \neq) \cap \I_n(\neq, <, -)|= {n+1 \choose 3} + 1
\]
and the result follows.
\end{proof}


\subsection{Class 187:  Conjectured to be counted by A049125}
\label{section:187}

It appears that the number of $e \in \I_n$ avoiding this pattern is given by A049125 in the OEIS,  where it is described by David Callan to be
the number of ordered trees with $n$ edges in which every non-leaf non-root vertex has at most one leaf child.
However, we have not yet proven it.   We can prove a characterization of the avoidance set and, from that, derive a 4-parameter recurrence that allows us to check against A049125 for several terms.

\begin{observation}
The sequences $e \in \I_n$ having no $i < j < k$ with $e_i \geq e_k$ are those for which $e_i > \max\{e_1, \ldots, e_{i-2}\}$ for $i=3, \ldots, n$.
For $e \in \I_n$, this is equivalent to the conditions $e_3 > e_1$ and,
for $4 \leq i \leq n$,  $e_i > \max\{e_{i-2},e_{i-3}\}$.
\end{observation}

\subsection{Class 193:  $\Sn_n(132,4312)$, $(n-1)2^{n-2} +1$}
\label{section:193}

Sequence $(n-1)2^{n-2} +1$ appears as A005183 in the OEIS, where Pudwell indicates that it counts 
$\Sn_n(132,4312)$ \cite{Sloane}.  We show it also counts $\I_n(\leq,\geq,\neq)$.

\begin{theorem}
$|\I_n(\leq,\geq,\neq)|=(n-1)2^{n-2}$
\end{theorem}
\begin{proof}
Observe that if $e \in \I_n$ has no $i < j < k$ such that $e_i \leq e_j \geq e_k$ and $e_i \not = e_k$
then $e$ must have the form
\[
e=(0, \ldots,0,e_a,0, \ldots, 0, e_{n-b+1}, e_{n-b+2}, \ldots, e_n\}
\]
where $1 \leq a < n+1$ and $b < n-a+2$ and $1 \leq e_a < e_{n-b+1}<e_{n-b+2}<\ldots<e_n<n$.

If $e_a>1$ then $e = 0 \cdot \sigma_1(e')$ for some $e' \in \I_{n-1}(\leq,\geq,\neq)$.
Otherwise, $e_a=1$ and  $e$ can be obtained by first choosing a $b$-element subset of 
$\{2, \ldots, n-1\}$ to place (sorted) in locations $n-b+1, \ldots, n$, and then choosing one of the locations
$2, \ldots, n-b$ to be the location $a$ such that $e_a=1$.  Thus the number of sequences containing a 1 is:
\[
\sum_{b=0}^{n-2} {n-2 \choose b}(n-1-b) = n 2^{n-3}.
\]

This gives the recurrence
\[
|\I_n(\leq, \geq,\neq)| = |\I_{n-1}(\leq,\geq,\neq)| + n 2^{n-3},
\]
where $|\I_1(\leq,\geq,\neq)| = 1$, whose solution is as claimed in the theorem.
\end{proof}

\subsection{Class 233:  $\I_n(012)$, $F_{2n-1}$}
\label{section:233}

It was shown by Corteel, et al.\ \cite{paisI} that the inversion sequences $e \in \I_n(<,<,-)=\I_n(012)$ are those  in which the positive elements  of $e$  are weakly decreasing.  From that characterization, it was proven that 
\[
|\I_n(<,<,-)|=|\I_n(012)| =F_{2n-1}.
\]
The sequence $F_{2n-1}$ also counts the {\em Boolean permutations}, given by $\Sn_n(321,3412)$ \cite{tenner,Ptenner}.

\subsection{Class 304: Conjectured to be counted by A229046}
\label{section:304}

We derive a recurrence  to count the $(-,-,=)$-avoiding inversion sequences.  This sequence appears to be
sequence A229046 in the OEIS.  If true, this would give a combinatorial interpretation of A229046 which so far is defined only by a generating function and summation.

Note that $\I_n(-,-,=)$ is the set of $e \in \I_n$ with at most two copies of any entry and any equal entries must be adjacent.  

Let 
$S_{n,k}$ be the set of $e \in \I_n(-,-,=)$ with $k$ distinct elements; that is, $S_{n,k}$ consists of the inversion sequences $e=(e_1,e_2,\ldots,e_n) \in \I_n(-,-,=)$ such that $|\{e_1, \ldots, e_n\}|=k$.  Let
$s(n,k)=|S_{n,k}|$.

\begin{theorem}
for $1 \leq k \leq n$,
\[
s(n,k) = (n-1+k)s(n-1,k-1)+(n-k)s(n-2,k-1),
\]
with initial conditions $s(1,1)=s(2,1)=s(2,2)=1$ and otherwise $s(n,k)=0$ for $k=1$ or $n \leq 2$.
\end{theorem}
\begin{proof}
Let $A_{n,k}$ be the subset of $S_{n,k}$ consisting of those $e$ in which $e_n$ is unrepeated.
Let $B_{n,k}$  = $S_{n,k} \setminus A_{n,k}$.
We can extend some $e \in S_{n,k}$ to strings in $S_{n+1,k}$ and $S_{n+1,k+1}$ in the following ways.

If $e \in B_{n,k}$, then $e \cdot n \in A_{n+1,k+1}$.  Additionally, if $x$ is one of the $n-k$ values in $\{0,1, \ldots, n-1\}$
not used in $e$, then $e \cdot x \in A_{n+1,k+1}$.

If $e \in A_{n,k}$, then $e \cdot n \in A_{n+1,k+1}$.
Furthermore, if $x$ is one of the $n-k$ values in $\{0,1, \ldots, n-1\}$
not used in $e$, then   $e \cdot x \in A_{n+1,k+1}$.
Finally, if $e_n=y$, then $e \cdot y \in B_{n+1,k}$.
Letting $a(n,k)=|A_{n,k}|$ and $b(n,k)=|B_{n,k}|$, we have
\begin{eqnarray*}
s(n,k) & = & a(n,k)+b(n,k);\\
b(n+1,k) & = & a(n,k);\\
a(n+1,k+1) & = & (n-k+1)b(n,k) + (n-k+1)a(n,k)\\
& =& (n-k+1)s(n,k).
\end{eqnarray*}
So,
\begin{eqnarray*}
s(n,k)  &= & a(n,k)+b(n,k)\\
& =& a(n,k)+a(n-1,k)\\
& =& (n-k+1)s(n-1,k-1) + (n-k)s(n-2,k-1).
\end{eqnarray*}
\end{proof}
Then $|\I_n(-,-,=)| = s(n,1) + \ldots + s(n,n)$. It is an open question to show that this theorem provides a refinement of A229046. Additionally, an interesting question is whether there is a natural description of the set $\Theta^{-1}( \I_n(-,-,=))$.

\subsection{Classes 429(A,B,C): Catalan numbers}
\label{section:429}
It is known that for any $\pi \in \Sn_3$, $|\Sn_n(\pi)|$ is the Catalan number $C_n = {2n \choose n}/(n+1)$ \cite{macmahon,SimionSchmidt}.
There are three inequivalent triples of relations $\rho = (\rho_1,\rho_2,\rho_3) \in \{\geq,\leq,<,>,=,\neq,-\}^3$ such that
$|\I_n(\rho)|=C_n$.  The first corresponds naturally under $\Theta: \Sn_n \rightarrow \I_n$ to a pattern $\pi \in \Sn_3$.

\subsubsection{Class {\bf 429A}:  Avoiding $e_j > e_k$}

\begin{theorem}
$\I_n(-,>,-) = \Theta(\Sn_n(213))$.
\end{theorem}
\begin{proof}
Observe that an $e \in \I_n$ has no $i < j < k$ with $e_j > e_k$ if and only if $e$ is weakly increasing.  Similarly,  it can be checked that $\pi \in \Sn_n$ avoids 213 if and only if $\Theta(\pi)$ is weakly increasing.
\end{proof}

\subsubsection{Class {\bf 429B}: Avoiding $e_j \geq e_k$ and $e_i <e_k$}

\begin{theorem}
$|\I_n(-,\geq,<)| = C_n$.
\end{theorem}
\begin{proof}
Observe that some $e \in \I_n$ has no $i < j < k$ with $e_i < e_k$ and  $e_j \geq e_k$ if and only if the positive elements of $e$ are strictly increasing.

Let $I(x) = \sum_{n \geq 0} \I_n(-,\geq,<) x^n$.  We will show that 
\begin{equation}
\label{equation:429B}
I(x)=1+xI^2(x),
\end{equation}
 which has the solution $\frac{1-\sqrt{1-4x}}{2x}$; recall that this is the generating function for $C_n$.

Given any $e \in \I_n(-,\geq,<)$, consider the last maximal entry $e_{t}$; this is the largest $t$ such that $e_t=t-1$.  The string $(e_1,e_2,\ldots,e_{t-1})$ is an element of $\I_{t-1}(-,\geq,<)$.  Additionally, it is straightforward to show that the string $\sigma_{1-t}(e_{t+1},e_{t+2},\ldots,e_{n})$ (where $t-1$ is subtracted from each positive value) is an element of $\I_{n-t}(-,\geq,<)$.  Conversely, any element of $\I_n(-,\geq,<)$ with last maximal entry in position $t$ is of the form $e' \cdot (t-1) \cdot \sigma_{1-t}(e'')$ where $e' \in \I_{t-1}(-,\geq,<)$ and $e'' \in \I_{n-t}(-,\geq,<)$.  This accounts for the $``xI^2(x)''$ term of equation \ref{equation:429B}.  Since this construction doesn't account for the length 0 inversion sequence, we must also add a ``1.''
\end{proof}

Alternatively, it can be checked that the following map from $ \I_n(-,>,-)$ to $\I_n(-,\geq,<)$ is a bijection.
Send $e \in \I_n(-,>,-)$ to $e'$, defined by
$e'_i=0$ if $e_i \in \{e_1, \ldots, e_{i-1}\}$ and otherwise $e'_i=e_i$.

\subsubsection{Class {\bf 429C}: Avoiding $e_i \geq e_j $   and $e_i \geq e_k$  }

In \cite{PAISII} we conjectured that  $|\I_n(\geq,-,\geq)|=C_n$.  This has been proven by Kim and Lin \cite{KimLin} where they also prove the following conjectures from the first version of this paper \cite{PAISII}:

\begin{itemize}

\item The number of $e \in \I_n(\geq,-,\geq)$ with
$\last(e)=k$ 
is equal to the number of
standard tableaux of shape $(n-1,k)$ (ballot numbers A009766 in OEIS \cite{Sloane}).



\item The number of $e \in \I_n(\geq,-,\geq)$ with $\dist(e)=k$
is equal to the number of $\pi \in \Sn_n(123)$ with $k-1$ descents (A166073 in the OEIS \cite{Sloane}).
(The number of {\em distinct elements} is $\dist(e) = |\{e_1, \ldots, e_n\}|$.)
\end{itemize}

\subsection{Class 523: $\Sn_n({\bar{3}}  {\bar{1}} 542)$ and the nexus numbers}
\label{section:523}

In this section we show that inversion sequences avoiding the pattern $(\neq,=, -)$ are equinumerous with permutations avoiding ${\bar{3}}  {\bar{1}} 542$. Note that a permutation  $\pi$ avoids the pattern ${\bar{3}}  {\bar{1}} 542$ if any occurrence of 542 in $\pi$ is contained in an occurrence of 31542.
We do this by proving  that  the $(\neq,=, -)$-avoiding inversion sequences with $k$ distinct entries are counted by the nexus numbers, $(n+1-k)^k - (n-k)^k$.

Observe that the sequences $e \in \I_n$ with no $i < j < k$ satisfying
$e_i \not = e_j = e_k$ are those in which the nonzero elements are distinct and once a nonzero element has occurred, at most one more 0 can appear in $e$.  We use this characterization to show that
$\I_n(\neq,=, -)$ is counted by the sequence A047970 which counts diagonal sums of nexus numbers \cite{Sloane}.  This sequence also counts permutations in $\Sn_n$ avoiding the barred pattern ${\bar{3}}  {\bar{1}} 542$ (conjectured by Pudwell, and proved by Callan \cite{callan11}).

Let $T_{n,k}$ be the set of $e \in \I_n(\neq,=, -)$ with $k$ distinct elements.  We prove the following refinement, which gives a new combinatorial interpretation of the {\em nexus numbers}, 
$(n+1-k)^k - (n-k)^k$
(see A047969 in the OEIS \cite{Sloane}).

\begin{theorem}
For $1 \leq k \leq n$, $|T_{n,k}| = (n+1-k)^k - (n-k)^k$.
\end{theorem}

\begin{proof}
We count $T_{n,k}$ directly.  When $k=1$, $|T_{n,k}|=1$ and the result follows.  When $k \geq 2$, any $e \in T_{n,k}$ will contain some $e_t$ such that $0=e_1=e_2=\ldots=e_{t-1}<e_t$ and there are no repeated values among $e_t,e_{t+1},\ldots,e_n$.  Therefore, if $e$ has $k$ distinct values, there are two cases: (1) $e$ begins with $n-k+1$ zeros and contains no other zeros; or (2) $e$ begins with $n-k$ zeros and contains one further zero after $e_{n-k+1}$ (which is the first nonzero entry).

For Case (1), the values $e_{n-k+2}, e_{n-k+3}, \ldots, e_n$ must all be distinct and nonzero.  So, there are $n-k+1$ possibilities for each, giving $(n-k+1)^{k-1}$ inversion sequences.  

For Case (2), $e_{n-k+1}$ must be nonzero, so there are $n-k$ choices for this entry.  Additionally, each of $e_{n-k+1},e_{n-k+2}, e_{n-k+3}, \ldots, e_n$ must be distinct, though zero could appear after $e_{n-k+1}$.  In total, this gives $(n-k)(n-k+1)^{k-1}$ possible inversion sequences.  Finally, we must remove any inversion sequence that does not include a zero among $e_{n-k+2}, e_{n-k+3}, \ldots, e_n$; there are $(n-k)^k$ such sequences.  As a result, there are $(n-k)(n-k+1)^{k-1}-(n-k)^k$ inversion sequences that are part of Case (2).

Adding Cases (1) and (2), we have $|T_{n,k}|=(n-k+1)^{k-1}+(n-k)(n-k+1)^{k-1}-(n-k)^k=(n-k+1)^{k}-(n-k)^k$, as desired.
\end{proof}

\subsection{Classes 772(A,B): Set partitions avoiding enhanced 3-cross\-ings}
\label{section:772}

In the first version of this paper \cite{PAISII}, we
conjectured that the avoidance sets for the patterns $(-,\leq,\geq)$ and $(\geq,\geq,-)$ are counted by A108307.  It was shown by Bousquet-M{\'e}lou and Xin that  A108307 gives the number of set partitions of $[n]$ avoiding enhanced 3-crossings \cite{BMX}.

The following theorem was proven by Lin through use of generating trees and the obstinate kernel method \cite{Lin} and by Yan through construction of a bijection with 0,1 fillings of Ferrers shapes \cite{Yan}.
Lin's proof makes use of our characterization of $I_n(\geq, \geq, -)$ in Observation \ref{observation:772B}.

\begin{theorem}[Lin \cite{Lin}, Yan \cite{Yan}]
$|I_n(\geq, \geq, -)|$ is the number set partitions of $[n]$ that avoid enhanced 3-crossings (or 3-nestings).
\end{theorem}

It follows that $(\geq,\geq,-)$ is indeed counted by A108307. We can  additionally show that there is  a bijection that not only proves Wilf-equivalence of the patterns
772A and 772B below, but also preserves a number of statistics. First consider the following characterizations of each class.

\begin{observation} 
The inversion sequences with no $i<j<k$ such that $e_i \geq e_j \geq e_k$ are precisely those
that can be partitioned into two increasing subsequences. 
\label{observation:772B}
\end{observation}
\begin{proof}
Suppose $e$ has such a partition $e_{a_1} < e_{a_2} < \cdots < e_{a_t}$ and
$e_{b_1} < e_{b_2} < \cdots < e_{b_{n-t}}$. 
 If there exists $i<j<k$ such that $e_i \geq e_j \geq e_k$, then no two of $i,j,k$ can both be in  $\{a_1, \ldots, a_t\}$ or both be in  $\{b_1, \ldots, b_{n-t}\}$, so $e$ avoids $(\geq,\geq,-)$. Conversely, if $e$ avoids $(\geq,\geq,-)$, let $a=(a_1, \ldots, a_t)$ be the sequence of left-to-right maxima of $e$.
Then $e_{a_1} < e_{a_2} < \cdots < e_{a_t}$.  Consider $i,j \not\in \{a_1, \ldots, a_t\}$ where $i<j$.  The fact that $e_i$ is not a left-to-right maxima implies there exists some $e_s$ such that $s<i$ and $e_s\geq e_i$.  Thus to avoid $(\geq,\geq,-)$, we must have $e_i < e_j$.
\end{proof}

\begin{observation}
Let $(e_1,e_2,\ldots,e_n) \in \I_n$.  Additionally, for any $i \in [n]$, let $M_i=\max(e_1,\allowbreak e_2,\ldots,e_{i-1})$.  Then $e \in \I_n(-,\leq,\geq)$ if and only if for every $i \in [n]$, the entry $e_i$ is a left-to-right maximum, or for every j where $i<j$,  we have $e_i > e_j$ or $M_i<e_j$.

\label{observation:772A}
\end{observation}
\begin{proof}
Let $e \in \I_n$ satisfy the conditions of Observation \ref{observation:772A} and, to obtain a contradiction, assume there exist $i<j<k$ such that $e_j\leq e_k$ and $e_i \geq e_k$ (that is $e_j \leq e_k \leq e_i$).  Notice that $M_j=\max\{e_1,e_2,\ldots,e_{j-1}\} \geq e_i$.  It follows that $M_j\geq e_k \geq e_j$, which contradicts our assumption.

Conversely, if $(e_1,e_2,\ldots,e_n) \in \I_n(-,\leq,\geq)$, consider any $e_i$.  If $e_i$ is not a left-to-right maximum, then there exists some maximum value $M_i=e_s$ such that $s<i$ and $e_s\geq e_i$.  Therefore, in order to avoid a 201 pattern, any $e_j$ where $j>i$ must have $e_i > e_j$ or $e_j > M_i=e_s$.
\end{proof}

\begin{theorem}
For $n \geq 1$, $|\I_n(\geq,\geq,-)|=|\I_n(-,\leq,\geq)|$.
\label{theorem:772A,B}
\end{theorem}
\begin{proof}
We exhibit a bijection based on the characterizations in Observations \ref{observation:772B} and \ref{observation:772A}.

Given $e \in \I_n(\geq,\geq,-)$, define $f \in \I_n(-,\leq,\geq)$ as follows.
Let $e_{a_1} < e_{a_2} < \cdots < e_{a_t}$ be the sequence of left-to-right maxima of $e$ and let
$e_{b_1} < e_{b_2} < \cdots < e_{b_{n-t}}$ be the subsequence of remaining elements of $e$.

For $i= 1, \ldots, t$, set $f_{a_i}=e_{a_i}$.  For each $j=1,2, \ldots, n-t$, we extract an element of the multiset $B=\{e_{b_1}, e_{b_2},\ldots, e_{b_{n-t}}\}$ and assign it to $f_{b_1}, f_{b_2},\ldots, f_{b_{n-t}}$ as follows: 
\[
f_{b_j} = {\rm max}\{k \ | \ k \in B - \{f_{b_1}, f_{b_2},\ldots, f_{b_{j-1}}\} \ {\rm and} \
k < {\rm max}(e_1, \ldots, e_{{b_j}-1}) \}.
\]
By definition, $f$ satisfies the characterization property in Observation \ref{observation:772A} of $\I_n(-,\leq,\geq)$.
\end{proof}

The bijection in Theorem~\ref{theorem:772A,B} preserves a number of statistics, as is shown below.

\begin{corollary}
All of the following statistics have the same distribution over $\I_n(-, \leq, \geq)$ and
$\I_n(\geq,\geq,-)$.

\begin{itemize}
\item the number of locations $i$ such that $e_i=i-1$;
\item the largest entry of $e$;
\item the number of zeros of $e$ (there can be at most two in either class);
\item the number of distinct elements of $e$ (and therefore the number of repeats in $e$);
\item the number of left-to-right maxima of $e$.
\end{itemize}
\end{corollary}

An interesting future direction would be to relate these statistics on the classes 772A and 772B of inversion sequences to corresponding statistics on set partitions avoiding enhanced 3-crossings.

\subsection{Classes 877(A,B,C,D):  Bell numbers and Stirling numbers}
\label{section:877}

The Bell number $B_n$ is the number of partitions of the set $[n]$ into nonempty blocks.  The Stirling number of the second kind, $S_{n,k}$ is the number of partitions of $[n]$ into $k$ blocks.

Among the triples of relations under consideration in this paper, four equivalence classes of patterns  have avoidance sets that appear to be counted by the Bell numbers.  We have shown this to be true for the classes 877A and 877C, whose inversion sequences have a similar character.  We have not confirmed this for the classes 877B and 877D, nor have we confirmed that 877B and 877D are Wilf-equivalent, even though our experiments show that there is likely a bijection that preserves several statistics.

It is interesting to note that $B(n)$ counts permutations avoiding $4 {\bar 1}32$ and several other barred patterns of length 4, as shown by Callan \cite{Callan06}.  Can any of these be related to one of the four patterns 877(A,B,C,D)?

\subsubsection{Class {\bf 877A}: Avoiding $e_i < e_j =e_k$}

These are the $011$-avoiding sequences.  
It was observed by Corteel, et al.\ \cite{paisI} that these are the $e \in \I_n$ in which the positive elements of $e$ are distinct, and that the $011$-avoiding sequences in $\I_n$ with $k$ zeros are counted by the Stirling number of the second kind, $S_{n,k}$. (This also appears in Stanley's \emph{Enumerative Combinatorics, Vol.\ 1} \cite{stanleyI}.) Thus $\I_n(011)$ is counted by the Bell numbers.  

\subsubsection{Class {\bf 877B}: Avoiding $e_i = e_j \geq e_k$}

The set $\I_n(=, \geq, -)$ consists of all $e \in \I_n$ such that no element appears more than twice and if an element, $x$, is repeated, all elements following the second occurrence of $x$ must be larger than $x$.  

From our calculations, it appears that $\I_n(=, \geq, -)$ is counted by the Bell numbers and, in fact, that the number of $e \in \I_n(=, \geq, -)$ with $k$ repeats is given by A124323, the number of set partitions of $[n]$ with $k$ blocks of size larger than 1, but we have not proven this.

\subsubsection{Class {\bf 877C}: Avoiding $ e_i \neq e_j \neq e_k$ and $e_i=e_k$}

Observe that these are the $e \in \I_n$ in which only adjacent elements of $e$ can be equal.
\begin{theorem}
The number of $e \in \I_n$ in which only adjacent elements of $e$ can be equal is $B_n$, the $n$th Bell number.
\end{theorem}
\begin{proof}
It can be checked that the following map from $ \I_n(\neq,\neq,=)$ to $\I_n(011)$ is a bijection.
Send $e \in \I_n(\neq,\neq,=)$ to $e'$, defined by
$e'_i=0$ if $e_i \in \{e_1, \ldots, e_{i-1}\}$ and otherwise $e'_i=e_i$.
\end{proof}

\subsubsection{Class {\bf 877D}: Avoiding $e_i \geq e_j$ and $e_i=e_k$}

$\I_n(\geq, -, =)$ is the set of $e \in \I_n$ such that no element appears more than twice and if an element  $x$ is repeated, all elements between the two occurrences of $x$  must be larger than $x$  (Note the similarity to 877B).

From our calculations, it appears that $\I_n(\geq, -, =)$ is also counted by the Bell numbers.  Moreover,  it appears that all of the following statistics are equally distributed over the classes 877B and 877D:

\begin{itemize}
\item the number of locations $i$ such that $e_i=i-1$;
\item the largest entry of $e$;
\item the number of zeros of $e$ (there can be at most two in either class);
\item the number of distinct elements of $e$ (and therefore the number of repeats in $e$, which appears to be A124323).
\end{itemize}

\subsection{Class 924:  Central binomial coefficients}
\label{section:924}

\begin{theorem}
$|\I_n(>,-,-)| = {2n-2 \choose n-1}$.
\end{theorem}
\begin{proof}
$\I_n(>,-,-)$ is the set of $e\in \I_n$ with $e_1 \leq \ldots \leq e_{n-1}$ (counted by the Catalan number $C_{n-1}$ as shown in Section \ref{section:429})
and with $e_n$ chosen arbitrarily from $\{0, \ldots, n-1\}$. Thus
\[
|\I_n(>,-,-)| = n C_{n-1} = n \left ( \frac{1}{n}  {2n-2 \choose n-1} \right ) =  {2n-2 \choose n-1} .
\]
\end{proof}

\subsection{Class 1064: Conjectured to be counted by A071356}
\label{section:1064}

Class 1064 are those inversion sequences with no $i<j<k$ such that $e_i > e_j \leq e_k$. These are the inversion sequences $e \in \I_n$ satisfying, for some $t$  such that $1 < t \leq n$,
\[
e_1 \leq \ldots \leq e_t > e_{t+1} > \ldots > e_n.
\]

Our experiments suggest that these are counted by A071356 in the OEIS, which Emeric Deutsch notes counts the number of underdiagonal lattice paths from $(0,0)$ to the line $x=n$ using only steps $R=(1,0)$, $V=(0,1)$, and $D=(1,2)$ \cite{Sloane}.

It also appears from our experiments that the distribution of the number of distinct elements of $e$ is symmetric and unimodal on 
$\I_n(<,\leq,-)$. The number of $e \in \I_n(<,\leq,-)$ with $\dist(e)=k$  is given in the table below for $n=1, \ldots 7$.
\begin{center}
\begin{tabular}{ccccccc}
1&&&&&&\\
1 & 1&&&&&\\
1 & 4 & 1&&&&\\
1 & 9 & 9 & 1 &&&\\
1 & 16 & 38 & 16 & 1 &&\\
1 & 25 & 110 & 110 & 25 & 1&\\
1 & 36 & 255 & 480 & 255 & 36 & 1 \\
\end{tabular}
\end{center}
If these observations are true in general, this provides a new simple combinatorial interpretation for 
 A071356 with a natural refinement via a symmetric statistic.

\subsection{Class 1265: $\Sn_n(2143,3142,4132)$ }
\label{section:1265}
\label{1265}

Observe that $\I_n(>,<,-)$ is the set of $e \in \I_n$ satisfying, for some $t$ with $1 < t \leq n$,
\[
e_1 \leq e_2 \leq \ldots \leq e_t > e_{t+1} \geq \ldots \geq e_n.
\]
Our experiments suggested that $\I_n(>,<,-)$ is counted by A033321, which counts $\Sn_n(\allowbreak2143,\allowbreak3142,\allowbreak4132)$, as well as permutations avoiding several other triples of 4-permutations.  Bur\-stein and Strom\-quist confirmed this by recognizing a natural bijection between $\Sn_n(2143,3142,\allowbreak 4132)$ and $\I_n(>,<,-)$ \cite{BS}.   Their theorem is as follows. Recall from Section 1 that  $invcode:\Sn_n \rightarrow \I_n$ is the reverse of the Lehmer code. 

\begin{theorem}[Burstein, Stromquist \cite{BS}]
For $n \geq 1$, $invcode(\Sn_n(2143,3142,4132)) = \I_n(>,<,-)$.
\end{theorem}

From Section 1, $invcode(\pi)=e$ if and only if $e=\Theta((\pi^{C})^{R})$, giving the following.

\begin{corollary}
$\Theta(\Sn_n(2143,3142,3241)) = \I_n(>,<,-)$. 
\end{corollary}

\subsection{Class 1347: Conjectured to be counted by $\Sn_n(4123,4132,4213)$}
\label{section:1347}

Our calculations suggest that $\I_n(>, -, \leq)$ is counted by A106228 in the OEIS, which was recently shown to count
$\Sn_n(4123,4132,4213)$ by Albert, Homberger, Pantone, Shar and Vatter \cite{AHPSV}.
We have not been able to confirm that our avoidance sequence is A106228. 

\subsection{Class 1385: $\I_n(000)$ and the Euler up/down numbers}
\label{section:1385}

$\I_n(=,=,-)$ is the set of inversion sequences avoiding the pattern ``000''.
It was shown by Corteel, et al.\ \cite{paisI} that
$|\I_n(000)| = E_{n+1}$, where $E_n$ is the Euler up/down number which counts the number of $\pi \in \Sn_n$ such that
$\pi_1 < \pi_2 > \pi_3 < \pi_4 > \cdots$.
The proof was via a bijection with $n$-vertex  0-1-2 increasing trees, which are also counted by $E_{n+1}$.

Another family of permutations counted by $E_{n+1}$ is the number of simsun permutations of $[n]$, introduced by Simion and Sundaram \cite{sundaram}.
A simsun permutation is one with no double descents, even after the removal of the elements 
$\{n,n-1, \ldots, k\}$ for any $k$.
It was shown that the number of $e \in \I_n(000)$ with $n-k$ distinct elements is the number of
simsun permutations of $n$ with $k$ descents \cite{paisI}.   The method of proof was to show that they satisfy the same recurrence. Recenty, Kim and Lin proved bijectively that the joint distribution of $\asc(\pi)$ and $\last(\pi)$  over simsun permutations $\pi$ of $[n]$ is the same as the  joint distribution of $\dist(e)$ and $\last(e)+1$ for $e \in \I_n(000)$.

An interesting question is whether there is a natural bijection between $\I_n(000)$ and up-down (or down-up) permutations of  $[n+1]$. For example, our calculations suggest that the number of $e \in \I_n(000)$ with $e_n=k-1$ is the number of down-up permutations $\pi$  of $[n+1]$ with $\pi_1=k+1$.  
\subsection{Class 1694: $\I_n(102)$}
\label{section:1694}

It was suggest by Corteel, et al.\ \cite{paisI} that $\I_n(>,-,<)$ is counted by A200753 in the OEIS \cite{Sloane}, a sequence defined by the generating function 
\begin{eqnarray}
A(x)& = &1 + (x-x^2)(A(x))^3.
\label{1694}
\end{eqnarray}
  This was confirmed by Mansour and Shattuck \cite{Mansour} who derived an explicit formula for $|\I_n(102)|$.

\begin{theorem}[Mansour-Shattuck  \cite{Mansour}]
The generating function 
$\sum_{n \geq 0} |\I_n(102)|x^n$ satisfies \eqref{1694}.
\end{theorem}

It would be interesting to find a direct combinatorial argument.

\subsection{1806(A,B,C,D):  large Schr\"oder numbers}
\label{section:1806}

The {\em  large Schr\"oder number}  $R_n$ is the number of  Schr\"oder $n$-paths; that is, the number of paths in the plane from $(0,0)$ to $(2n,0)$ never venturing below the $x$-axis, and using only the steps $(1,1)$ (up), $(1,-1)$ (down) and $(2,0)$ (flat).

In the area of pattern avoiding permutations,  $R_{n-1}$
counts the separable permutations $\Sn_n(2413,3142)$, as well as
$\Sn_n(\alpha,\beta)$ for many other pairs $(\alpha,\beta)$ of patterns of length 4 \cite{kremer}.  We have four inequivalent triples of relations whose avoidance sets are counted by the large Schr\"oder numbers, two  of which (1806B and 1806D) correspond in natural ways to
a pair of patterns of length 4.

\subsubsection{Class {\bf  1806A}: Avoiding $e_j > e_k$ and $e_i < e_k$}

These are the sequences avoiding $021$.  It was shown by Corteel, et al.\ \cite{paisI} that
 $e\in \I_n$ avoids $021$ if and only if its positive entries are weakly increasing.
It was also proven in \cite{paisI} that $|\I_n(021)|=R_{n-1}$.

The following refinements were shown:
\begin{itemize}
\item The number of $021$-avoiding inversion sequences  $e$ in $\I_n$ with $k$ positions $i$ such that $e_i=i-1$ is equal to the number of Schr\"oder $(n-1)$-paths with $k-1$ initial up steps.

\item The number of $021$-avoiding inversion sequences  $e$ in $\I_n$ with $k$ zeros is equal to the number of Schr\"oder $(n-1)$-paths with $k-1$ peaks (or $k-1$ flat steps).
\end{itemize}

It was also shown that the ascent polynomial for $\I_n(021)$  is palindromic and corresponds to sequence A175124 in the OEIS.

\subsubsection{Class {\bf  1806B}: Avoiding $e_i > e_j$  and $e_i \geq e_k$ }

It is known that $\Sn_n(2134,2143)$ is counted by $R_{n-1}$ \cite{kremer}.  This is a member of ``Class VI'' in Kitaev's book \cite{kitaev};
we use this fact to count $\I_n(>, -, \geq)$.
\begin{theorem}
$|\I_n(>, -, \geq)|= R_{n-1}.$ 
\end{theorem}

\begin{proof}
We show that $\Theta(\Sn_n(2134,2143))=\I_n(>, -, \geq)$.

Let  $e \in \I_n$ satisfy $e_i > e_j$ and $e_i \geq e_k$ for some $i < j < k$.  Let $\pi = \Theta^{-1}(e)$.
Then $\min\{\pi_j,\pi_k\} > \pi_i$ and, since $e_i>e_j$,
 there must exist $a < i$ such that both  $\pi_a > \pi_i$ and  $\min\{\pi_j,\pi_k\} > \pi_a$. 
 Thus $\pi_a\pi_i\pi_j\pi_k$ forms either a $2134$ or a $2143$.

 Conversely, suppose, for some $\pi \in \Sn_n$, that  $\pi_a\pi_i\pi_j\pi_k$ is one of the patterns $2134$ or $2143$ and let
 $e = \Theta(\pi)$.
 Let $j'$ be the smallest index larger than $i$ for which $\pi_{j'} > \pi_a$.
 Then $\pi_{i+1}, \ldots, \pi_{j'-1}$ are all smaller than $\pi_a < \pi_{j'}$ and so $e_i>e_{j'}$.
 Let $k'$ be the smallest index larger than $j'$ such that $\pi_{k'} > \pi_a$.
 Then, with the possible exception of $\pi_{j'}$, all of $\pi_{i+1}, \ldots, \pi_{k'-1}$ are smaller than $\pi_{k'}$ (since these entries are necessarily smaller than $\pi_a$).  In addition, since
 $\pi_i<\pi_a < \pi_{k'}$, we have  $e_i \geq e_{k'}$. Thus $e$ has the pattern $(>,-,\geq)$.
 \end{proof}
 
From our calculations, it appears that the ascent polynomial for $\I_n(>,-,\geq)$ is the same as that for 1806A, which was palindromic.  This has been confirmed by Kim and Lin \cite{KimLin}.
Since $\Theta$ sends descents to ascents, this implies that the descent polynomial for  $\Sn_n(2134,2143)$ is palindromic.  This is not true in general for permutations avoiding pairs of patterns of length 4, even those counted by the large Schr\"oder numbers.  For example, it is not true of
$\Sn_n(1234,2134)$ or $\Sn_n(1324,2314)$.  However Fu, Lin, and Zeng have recently shown that
the descent polynomial for the separable permutations $\Sn_n(2413,3142)$ is $\gamma$-positive and therefore
palindromic \cite{FuLinZeng}.

\subsubsection{Class {\bf  1806C}: Avoiding $e_i \geq e_j$ and $e_i >e_k$}

\begin{theorem}
$|\I_n( \geq, -, >)|=R_{n-1}$.
\end{theorem}

\begin{proof}

We will construct a generating function for $\I_n( \geq, -, >)$.

Let $E(x) = \sum_{i = 1}^{\infty} |\I_n( \geq, -, >)|x^n$.  We will show that $E(x)$ satisfies \[E(x) = x + xE(x) + E^2(x),\] whose solution is \[E(x) = \frac{1-x-\sqrt{x^2-6x+1}}{2}.\]  This implies that $|\I_n( \geq, -, >)|$ is the $(n-1)$th large Schr\"{o}der number.

Let $e = e_1e_2 \ldots e_n \in \I_n( \geq, -, >)$.  Let $e_{t}$ be the latest maximal entry of $e$; that is, $\max \{i \mid e_i = i-1\}$.  If $t = 1$, then either $e = (0)$ or $e = (0, f_1, f_2, \ldots, f_{n-1})$ for some $(f_1, f_2, \ldots f_{n-1}) \in \I_{n-1}( \geq, -, >)$.

Now consider the case where $t>1$.  Notice that $e_{t+1} \leq t-1$ since $e_{t+1}$ cannot be maximal.  This implies that $(e_1,e_2 ,\ldots, e_{t-1},e_{t+1})$ is an inversion sequence of length $t$.  Furthermore, it is straightforward to show that $(e_1,e_2, \ldots, e_{t-1},e_{t+1}) \in \I_n( \geq, -, >)$.  

Additionally, for all $e_j$ where $j >t+1$, we must have $e_j \geq t-1$; if this is not the case, and there exists some $j$ where $e_j <t-1$, then we have $t<t+1<j$ where $e_t \geq e_{t+1}$ and $e_t>e_j$.  Therefore $(e_t - t+1, e_{t+2}-t+1, \ldots, e_n - t+1) \in \I_{n-t}( \geq, -, >)$.

Conversely, for any sequences $(e_1, e_2, \ldots, e_t) \in \I_t( \geq, -, >)$ and $(f_1, f_2, \ldots, f_{n-t}) \in \I_{n-t}( \geq, -, >)$, we can construct the inversion sequence \[(e_1, e_2, \ldots, e_{t-1}, f_1+t-1, e_t, f_2+t-1, f_3+t-1, \ldots, f_{n-t}+t-1).\]  It is straightforward to show that this inversion sequence is in $\I_n( \geq, -, >)$ and the last maximal entry is in the $t$-th position.
\end{proof}


\subsubsection{Class {\bf  1806D}: Avoiding $e_i \geq e_j \not = e_k$ and $e_i \geq e_k$}

Due to a result of Kremer \cite{kremer}, we know that $|\Sn_n(4321,4312)|=R_{n-1}$.  We can prove that $|\I_n(\geq, \neq, \geq)|$ is also enumerated by $R_{n-1}$ by constructing a bijection between $\Sn_n$ and $\I_n$ that restricts to a bijection between $\Sn_n(4321,4312)$ and $\I_n(\geq, \neq, \geq)$.  This bijection is useful for another class of inversion sequences: it will be used later to prove results related to $\I_n(>,\neq,>)$ (classified as 3720).

\begin{definition} \label{def phi}
Let $\pi \in S_n$ and define $\phi(\pi) = e_1e_2\ldots e_n \in \I_n$ as follows, starting with $e_n$ and defining entries in reverse order.

\begin{enumerate}
\item $e_n = \pi_n-1$
\item For $1\leq i <n$, 
	\begin{enumerate}
		\item if $\pi_i \leq i$, then $e_i = \pi_i-1$.  
		\item otherwise, if $\pi_i$ is the $k$-th largest element of $\{\pi_1, \ldots, \pi_i\}$
       then $e_i$ is the $k$-th smallest element of the set $\{e_{i+1}, ..., e_n\}$.
	\end{enumerate}
\end{enumerate}
\end{definition}

\begin{lemma}
For $\pi \in \Sn_n$, $\phi(\pi) \in \I_n$.
\end{lemma}

\begin{proof}
To show that $\phi(\pi) \in \I_n$, we need to prove that $0 \leq e_i \leq i-1$ for every $i \in [n]$.  We will use an inductive argument, starting with $e_n$, to show this.  We defined $e_n=\pi_n-1$; since $1 \leq \pi_n \leq n$, it follows that $0 \leq e_n \leq n-1$, as desired.  Now consider $e_i$ and assume that for all $e_j$ among $e_{i+1}e_{i+2}\ldots e_n$, $0 \leq e_j \leq j-1$.  If $\pi_i \leq i$, then $e_i = \pi_i - 1$ and it immediately follows that $0 \leq e_i \leq i-1$.  

If instead $\pi_i > i$, assume that $\pi_i$ is in the $k$-th largest element of $\{\pi_1, \ldots, \pi_i\}$.  Notice that each value of $\{e_{i+1}, ..., e_n\}$ corresponds to an entry $\pi_j$ where $i<j$ and $\pi_j\leq j$ (any entry $\pi_j$ where $\pi_j>j$ will repeat a value).  So, there are $n-\pi_i-k+1$ entries $\pi_j$ such that $i<j$ and $\pi_j>\pi_i$; in turn, this implies that there are $(n-i)-(n-\pi_i-k+1)=\pi_i-(i+1)+k$ entries $\pi_j$ such that $i<j$ and $\pi_j<\pi_i$.  At a maximum, $\pi_i-(i+1)$ of these entries are greater than $i$; this leaves $k$ entries occurring after $\pi_i$ that are less than or equal to $i$.  Each of these entries corresponds to a value in $\{e_{i+1}, ..., e_n\}$ that is less than or equal to $i-1$.  It follows, that $e_i \leq i-1$, as desired.
\end{proof}

\begin{lemma}
$\phi: \Sn_n \rightarrow \I_n$ is a bijection.
\end{lemma}

\begin{proof}
Let $e = e_1 e_2 \ldots e_n \in \I_n$.  We can define the inverse image of $e$, $\phi^{-1}(e) = \pi_1\pi_2\ldots\pi_n$ in reverse order, starting with $\pi_n$ so that $\pi_n = e_n+1$.  For $1\leq i <n$, if $e_i \neq e_j$ for all $j$ where $i<j \leq n$, then $\pi_i = e_i+1$; otherwise, if $e_i$ is the $k$-th smallest value of $\{e_{i+1}, \ldots, e_n\}$, $\pi_i$ is the $k$-th largest value of $[n]$ that does not appear among $\pi_{i+1}, \ldots,\pi_n$.
\end{proof}

It is interesting to note that for any $\pi \in \Sn_n$, $\exc(\pi)=\repeats(\phi(\pi))$, where $\exc(\pi)$ is the number of positions $i$ such that $\pi_i > i$.

Now, we show that $\phi$ restricts to a bijection between $\Sn_n(4321,4312)$ and $\I_n(\geq, \neq, \geq)$ by proving the following:

\begin{theorem} \label{thm:1806D}
$\phi(\Sn_n(4321,4312))=\I_n(\geq, \neq, \geq)$
\end{theorem}

\begin{proof}
Consider some $\pi \in \Sn_n$ that contains an occurrence of 4321 or 4312.  So, there exists some $a<i<j<k$ such that $\pi_a\pi_i\pi_j\pi_k$ form a 4321 or 4312 pattern.  We will show that there exists an occurrence of the pattern $(\geq, \neq, \geq)$ in $\phi(\pi) = e$.

We must consider two cases.  If $\pi_i\leq i$, then $j >\pi_i>\pi_{j}$ and $k>\pi_i>\pi_{k}$.  Therefore, $e_i = \pi_i-1$, $e_{j} = \pi_{j}-1$ and $e_{k} = \pi_{k} -1$.  So, since $\pi_i>\max\{\pi_{j},\pi_{k}\}$ and $\pi_{j} \neq \pi_{k}$,  $e_i, e_{j}, e_{k}$ forms an occurrence of $(\geq, \neq, \geq)$.

Now assume $\pi_i> i$.  Recall that $e_i$ is the $t$-th smallest element of $\{e_{i+1},\ldots,e_n\}$ if $\pi_i$ is the $t$-th largest element of $\{\pi_1,\pi_2,\ldots,\pi_i\}$.   Since $\pi_a$ is larger than and occurs before $\pi_i$, we know that $t$ is at least 2.  So, if $e_{j'}$ and $e_{k'}$, where $j' < k'$, are the two smallest distinct values in the set $\{e_{i+1},\ldots,e_n\}$, we are guaranteed that $e_i \geq e_{j'}$ and $e_i \geq e_{k'}$; so, $e_i, e_{j'}, e_{k'}$ form the pattern $(\geq, \neq, \geq)$.

For the converse, let $\pi \in \Sn_n$ such that $e= \phi(\pi)$ has indices $i < j < k$ such $e_i \geq e_j$, $e_i \geq e_k$, and $e_i \not = e_k$.  We show that $\pi$ contains 4321 or 4312.

If $e_i \in \{e_{i+1}, \ldots, e_n\}$ then,  by definition of $\phi$, $\pi_i > i$ and $e_i$ is the $k$th smallest element of $\{e_{i+1}, \ldots, e_n\}$ for some $k \geq 2$, noting that  the distinct elements $e_j$ and $e_k$ are both at most $e_i$.  Then, again by definition of $\phi$, $\pi_i$ is the $k$th largest element of $\{\pi_1, \ldots, \pi_i\}$, so there is some $a < i$ with $\pi_a > \pi_i$.  To show that $\pi$ contains one of the patterns 4321 or 4312, it remains to show there are at least two elements of $\{\pi_{i+1}, \ldots, \pi_n\}$ that are smaller than $\pi_i$.  The number of elements in $\pi$ that are larger than $\pi_i$ is $n-\pi_i < n-i$.  At least one of these is to the left of $\pi_i$ in $\pi$.  Thus at most $n-i-2$ can be in the $n-i$ positions to the right of $\pi_i$.

If $e_i  \not \in \{e_{i+1}, \ldots, e_n\}$ then,  by definition of $\phi$, $\pi_i \leq i$ and $e_i = \pi_i -1$.
Since $e_j \leq e_i \leq i < j$ and $e_k \leq e_i \leq i < k$, we have $e_j = \pi_j -1$ and $e_k = \pi_k -1$, so that $\pi_i\pi_j\pi_k$ has the pattern 321 or 312.  To show that $\pi$ contains 4321 or 4312, it remains to show there is an $a<i$ such that $\pi_a > \pi_i$.  Note that there are $n-\pi_i  \geq n-i$ elements of $\pi $ larger than $\pi_i$.  The $n-i$ positions to the right of $\pi_i$ hold at least two elements smaller that $\pi_i$.  So there must be a larger element to the left of $\pi_i$, that is, an $a<i$ such that $\pi_a > \pi_i$.

\end{proof}

Our calculations suggested that the ascent polynomial for the inversion sequences in $\I_n(\geq, \neq, \geq)$ is the same as the (symmetric) ascent polynomial for 1806A.
This has been proven by Kim and Lin in \cite{KimLin}.
It also appears that the number of these inversion sequences with $k$ ``repeats'' is counted by A090981, the
number of Schr\"oder paths with $k$ ascents.

\subsection{Class 2074: Baxter numbers}
\label{section:2074}


In the earlier version of this paper \cite{PAISII}, it was conjectured that the avoidance sequence for the pattern $(\geq, \geq, >)$ is A001181 in the OEIS. This sequence counts the Baxter permutations, which is a result of Chung, et al.\ in \cite{chung}. A {\em Baxter permutation} $\pi$ is one that avoids  the vincular patterns 3\underline{14}2 and 2\underline{41}3; that is, there 
is no $i < j < k$ such that $\pi_j < \pi_k <  \pi_i <\pi_{j+1}$ or  $\pi_j > \pi_k > \pi_i > \pi_{j+1}$.

Kin and Lin prove this via the so-called obstinate kernel method \cite{KimLin}:
\begin{theorem}[\cite{KimLin}]
$|I_n(\geq,\geq,>)|$ is the number of Baxter permutations of $[n]$.
\end{theorem}

Note that the Baxter permutations contain the separable permutations $\Sn_n(3142,2413)$ which are counted by the large Schr\"oder numbers.  Similarly, $\I_n(\geq,\geq,>)$ contains the inversion sequences  $\I_n(\geq, -, >)$ (which define class 1806C) which are also counted by the  large Schr\"oder numbers. 
It would be nice to find a bijection between $(\geq,\geq,>)$-avoiding inversion sequences  and Baxter permutations that restricts to a bijection between $(\geq,-,>)$-avoiding inversion sequences  and separable permutations.

\subsection{Classes 2549(A,B,C):  Conjectured to be counted by \allowbreak$\Sn_n(4231,\allowbreak42513)$}
\label{section:2549}

Albert, et al.\ \cite{albert2549} showed that
\[
|\Sn_n(4231,42513)| = \sum_{i=0}^n \frac{n-i}{2i+n} {2i+n \choose i},
\]
which is sequence A098746 in the OEIS \cite{Sloane}. It appears from our calculations that the inequivalent classes $\I_n(>,-,>)$, $\I_n(>, \neq, \geq)$, and $\I_n(\geq, \neq, >)$ have the same avoidance sequence A098746.  We will show at least that the classes 2549A ($\I_n(>,-,>)$) and 2549C ($\I_n(\geq, \neq, >)$) are Wilf-equivalent.

\begin{theorem}
The patterns $(>,-,>)$ and $(\geq, \neq, >)$, defining classes {\rm 2549A} and {\rm 2549C} respectively, are Wilf-equivalent.
\label{thm:2549AC}
\end{theorem}
The proof relies on Theorem \ref{thm:2958BC} in the next section and will be given there.

\subsection{Classes 2958(A,B,C,D): Plane permutations and the semi-Baxter numbers}
\label{section:2958}

In an earlier version of this paper \cite{PAISII}, we conjectured that all four pattern equivalence classes 2958(A,B,C,D) have avoidance sets equinumerous with the \emph{plane permutations}, $\Sn_n(21\bar{3}54)$.  These are permutations in which every occurrence of the pattern 2154 is contained in an occurrence of 21354. We showed that the classes 2958(B,C,D) are Wilf-equivalent (Theorems \ref{thm:2958BC} and \ref{2958BD} below); additionally, we provided a characterization of the inversion sequences avoiding 2958B (Observation \ref{characterize2958B} below) and proved a recurrence for the avoidance sequence of 2958B.  

We can now make use of a recent paper of Bouvel, Guerrini, Rechnitzer, and Rinaldi \cite{BouvelGuerriniEtAl} to confirm our conjecture.

Bouvel, et al.\ coined the term ``semi-Baxter'' for the sequence of integers that enumerates the class of permutations avoiding the vincular pattern $2 \underline{41} 3$ (an occurrence of $2 \underline{41} 3$ is an occurrence of $2413$ where the ``4'' and ``1'' occur consecutively) \cite{BouvelGuerriniEtAl}. They derived a functional equation and used it to get a closed-form formula for the semi-Baxter sequence.

Moreover, Bouvel, et al.\  showed that semi-Baxter permutations are equinumerous with plane permutations:  they  used generating trees and proved that both classes of permutations are generated by the same succession rule $\Omega_{semi}$. Therefore, their formula for semi-Baxter numbers also enumerates plane permutations, which fulfills a challenge posed by Bousquet-M{\'e}lou and Butler \cite{BMB}.

Then, using our characterization of the inversion sequences avoiding 2958B,
Bouvel, et al.\  showed that those inversion sequences are also generated by the succession rule $\Omega_{semi}$ and are therefore equinumerous with plane and semi-Baxter permutations. This result, together with our proof below of the Wilf-equivalence of 2958(B,C,D), establishes that all three patterns give rise to avoidances sets that are equinumerous with plane permutations.

To complete the picture and confirm our original conjecture, we will show that inversion sequences avoiding the pattern 2958A also grow according to the succession rule $\Omega_{semi}$.

\subsubsection{Class {\bf 2958A}: Avoiding $e_j < e_k$ and $ e_i\geq e_k$}

Bouvel, et al.\ utilize generating trees to relate distinct combinatorial objects to the semi-Baxter sequence, providing a \emph{succession rule} for the sequence. To use this succession rule, each combinatorial object is identified with a label, $(h,k)$, which is based on the properties of the object. The rule then describes the labels of those members of the same class of objects that can be obtained by systematically adding  one atom.  For our purposes, adding one atom will correspond to adding a new last entry to an inversion sequence. 

The semi-Baxter sequence follows the succession rule:

\[
\Omega_{semi} = \begin{cases} (1,1)& \\ (h,k)\rightsquigarrow&(1,k+1),\ldots,(h,k+1) \\ &(h+k,1),\ldots,(h+1,k).   \end{cases}
\]

The top line of the succession rule gives the label for the ``root'' object (which will be the inversion sequence (0) in our case) and the second line shows the labels of the inversion sequences obtained by systematically adding a new last entry.

We show how to define the labels $(h,k)$ for inversion sequences avoiding $(-,<,\geq)$ in such a way that those sequences grow according to the succession rule $\Omega_{semi}$.

\begin{theorem}
$\I_n(-,<,\geq)=\I_n(101,201)$ grows according to the succession rule $\Omega_{semi}$ and is therefore enumerated by the semi-Baxter sequence.
\end{theorem}

\begin{proof}
For each  $e \in \I_n(-,<,\geq)$, we will account for  all  $e' \in \I_{n+1}(-,<,\geq)$ obtainable by appending a new last entry $e_{n+1}$. Define an \emph{active site} for $e$ to be a value $c$ such that $(e_1,e_2,\ldots,e_n,c) \in \I_{n+1}(-,<,\geq)$. Now, for any inversion sequence $e$, let $h$ denote the number of active sites less than or equal to $\maxx(e)$, and let $k$ denote the number of active sites greater than $\maxx(e)$. 
 First notice that for $e=(0) \in \I_1(-,<,\geq)$, since $\maxx(e)=0$, the set of active sites is $\{0,1\}$, with $0 \leq \maxx(e)$ and $1 > \maxx(e)$, so $(h,k) = (1,1)$. 

For $n\geq 2$, let $e \in \I_n(-,<,\geq)$.  Let $e_{n+1}$ be the value of one of the active sites  for $e$ and consider $e'=(e_1,\ldots,e_{n+1}) \in \I_{n+1}(-,<,\geq)$.

If $e_{n+1} \leq \maxx(e)$, then since $e'$ must only avoid 101 and 201, an active site for
$e'$ must either be  less than or equal to $e_{n+1}$ or greater than $\maxx(e)=\maxx(e')$. So, the labels of the inversion sequences we obtain by appending such  an active site  $e_{n+1}$ to $e$  are $(1,k+1), (2,k+1), \ldots, (h, k+1)$.

If $e_{n+1} > \maxx(e)$, again since $e'$ must only avoid 101 and 201, all sites active for $e$ are also active for $e'$. Additionally, $e'$  will have a new active site: $n+1$. Since $\maxx(e') > \maxx(e)$,  some active sites for $e$ that were counted with $k$ will now be counted with $h$ when finding the label for $e'$.  The labels of the inversion sequences we obtain by appending 
such  an active site  $e_{n+1}$ to $e$ 
 are $(h+k,1),(h+k-1,2),\ldots,(h+1,k)$.

It follows that $\I_n(-,<,\geq)$ grows by the succession rule $\Omega_{semi}$ and is therefore enumerated by the semi-Baxter sequence.
\end{proof}

\subsubsection{Class {\bf 2958B}: Avoiding $e_i > e_j \geq e_k$}

Note that $\I_n(>,\geq,-) = \I_n(210,100)$.  It was shown by Corteel, et al.\ \cite{paisI} that the inversion sequences avoiding 210 (class 4306A) have the following useful characterization.  

Define a {\em weak left-to-right maximum} in an inversion sequence $e$ to be a position $j$ such that $e_i \leq e_j$ for all $i \in [j-1]$.

For $e \in \I_n$, let $a_1 \leq a_2 \leq \cdots \leq a_t$ be the sequence of weak left-to-right maxima of $e$.
Let $b_1 < b_2 < \cdots < b_{n-t}$ be the sequence of remaining indices in $[n]$.
Let $e^{top}=(e_{a_1}, e_{a_2}, \ldots, e_{a_t})$ and $\tp(e)=e_{a_t}$.
Let $e^{bottom} = (e_{b_1},e_{b_2}, \ldots, e_{b_{n-t}})$ and $\bottom(e)=e_{b_{n-t}}$.  If every entry of $e$ is a weak left-to-right maximum, then $e^{bottom}$ is empty and we set $\bottom(e)=-1$.

\begin{observation} [\cite{paisI}]  The inversion sequence $e$ avoids 210 if and only if 
$e^{top}$ and $e^{bottom}$ are weakly increasing sequences.
\end{observation}

We can extend this to $\I_n(210,100)$.
\begin{observation}    The inversion sequence $e$ avoids both 210 and 100 if and only if 
$e^{top}$ is weakly increasing and and $e^{bottom}$ is strictly increasing.
\label{characterize2958B}
\end{observation}

Using Observation \ref{characterize2958B},
Bouvel, et al.\ showed the following (where $SB_n$ denotes the $n$th semi-Baxter number):

\begin{theorem}[\cite{BouvelGuerriniEtAl}. Theorem 20]
There are as many inversion sequences of size $n$ avoiding 210 and 100 as plane permutations of size $n$. In other words $|\I_n(210,100)| = SB_n$.
\end{theorem}

The characterization in Observation~\ref{characterize2958B} can be used to derive a recurrence.

\begin{theorem} 
Let $S_{n,a,b}$ be the number of $e \in \I_n(201,100)$ with $\tp(e) = a$ and  $\bottom (e)= b$.
Then
\begin{equation*}
S_{n,a,b} = \sum_{i=-1}^{b-1}S_{n-1,a,i} +  \sum_{j=b+1}^{a}S_{n-1,j,b},
\label{Snab:recurrence}
\end{equation*}
with initial conditions 
$S_{n,a,b}=0$ if $a \geq n$ and $S_{n,a,-1}= \frac{n-a}{n}{n-1+a \choose a}$.
\label{rec:210and100}
\end{theorem}

From Theorem \ref{rec:210and100} we get:
\begin{equation}
|\I_n(210,100)|\ = \ \sum_{a=0}^{n-1} \sum_{b=-1}^{a-1} S_{n,a,b} \ =\  \frac{1}{n+1}\binom{2n}{n}  +\sum_{a=0}^{n-1} \sum_{b=0}^{a-1} S_{n,a,b}.
\label{count210}
\end{equation}

\subsubsection{Class {\bf 2958C}: Avoiding $e_i \geq e_j > e_k$}

Now we will show that the class 2958B is Wilf-equivalent to 2958C.

\begin{theorem}
The patterns $(-,<,\geq)$ and $(\geq,>,-)$, defining classes 2958B and 2958C respectively, are Wilf-equivalent.
\label{thm:2958BC}
\end{theorem}

\begin{proof}
The avoidance sets for classes 2958B and 2958C are  $\I_n(100,210)$ and $\I_n(110,210)$, respectively.
We describe a bijection
\[
\alpha: \I_n(110,210) \rightarrow \I_n(100,210).
\]
For $e \in \I_n(110,210)$, let $\alpha(e)=e' = (e_1', \ldots, e_n')$ where, for $1 \leq j \leq n$,
\begin{equation}
e_j' = \left \{
\begin{array}{ll}
\max \{e_1, \ldots, e_j\}, & {\rm if} \  \ e_j=e_k \ {\rm for \  some} \ k > j; \\
e_j, & {\rm otherwise.}
\end{array}
\right .
\label{alpha}
\end{equation}
Note that for $1 \leq j \leq n$, 
\begin{equation}
e'_j \ \leq \ \max\{e_1, \ldots, e_j\} \ = \ \max \{e'_1, \ldots, e'_j\}
\label{alphaprop1}
\end{equation}
and
if $e'_j \neq e_j$ then 
\begin{equation}
 e_j'  \ = \  \max\{e_1, \ldots, e_j\} \ \geq e'_i \ {\rm for} \ 1 \leq i < j.
\label{alphaprop2}
\end{equation}
To see that $e'$ avoids 100, suppose $e_i' > e_j' = e_k'$ for some $i < j < k$. Then by \eqref{alphaprop2}, we must have $e_j'=e_j$ and $e_k'=e_k$.  But from the definition of $\alpha$
\eqref{alpha}, if $e_j=e_k$ where $j<k$ then $e_j' = \max\{e_1, \ldots, e_j\} \geq e'_i$, which is a contradiction.

To see that $e'$ avoids 210, suppose that $e_i' > e_j' > e_k'$ for some $i < j < k$.  Again by 
 \eqref{alphaprop2},
$e_j'=e_j$ and $e_k'=e_k$.  Since $e$ avoids 210, $e_i' \neq e_i$ so $e_i'= \max\{e_1, \ldots, e_i\} = e_s$ for some $s \in [i]$.  But then $e_s > e_j > e_k$ is a 210 in $e$.

Thus $e' = \alpha(e) \in \I_n(100,210)$.  To show that $\alpha$ is a bijection, we define its inverse $\beta$.

First we make an observation. 
Consider some $e \in \I_n(110, 210)$ and an entry $e_j$ such that $e'_j \neq e_j$ in $\alpha(e) = e'$.  This implies that there exists some index $k$ with $j<k$ and $e_j=e_k=m$ for some value $m$.  Additionally, it must be the case that $m < M = \max\{e_1, \ldots, e_{j-1}\} = e_i$ where $i \in [j-1]$ (else, $e'_j=m$).
Then, since $e$ avoids 210, 
we must have $m= \min\{e_j, \ldots, e_n\}$.  
Thus, for $e'= \alpha(e)$ we have
\[e_i'=e_i=M=e_j'>e_j=m=\min \{e_j , \ldots, e_n\}=\min \{e'_j , \ldots, e'_n\}.
\]
So, we can reconstruct $e$ from $e'$ by defining $\beta: \I_n(100,210) \rightarrow \I_n(110,210)$
 as follows.

For $e \in \I_n(100,210)$, let $\beta(e)=e' = (e_1', \ldots, e_n')$ where for $1 \leq j \leq n$
\[
e_j' = \left \{
\begin{array}{ll}
\min \{e_j, \ldots, e_n\}, & {\rm if} \  \ e_i=e_j \ {\rm for \  some} \ i< j; \\
e_j, & {\rm otherwise.}
\end{array}
\right .
\]
Then $\beta(\alpha(e))=e.$  We can check similarly that $\beta(e) \in \I_n(110,210)$ and $\alpha(\beta(e))=e$.
\end{proof}

We now return to the proof of the Wilf-equivalence of the classes 2549A and 2549C from the previous section.

\begin{proof}[Proof of Theorem~\ref{thm:2549AC}]

Observe that  2549C $\preceq$ 2958C and 2549A $\preceq$ 2958B in the following sense:
\[
\rm{2549C} : \I_n(\geq, \neq,>) = \I_n(110,210,201) \subseteq \I_n(110,210) = \I_n(\geq, >,-) : \rm{2958C}
\]
\[
\rm{2549A} : \I_n(> ,-,>) = \I_n(100,210,201) \subseteq \I_n(100,210) =\I_n(>, \geq, -) : \rm{2958B}
\]
We check that the mapping $\alpha$  \eqref{alpha} restricts to a bijection between inversion sequences in class 2549C and in class 2549A. 

Let $e \in \I_n(110,210)$ and let $e'=\alpha(e)$.  Suppose $e$ avoids the  pattern 201, but that for some $i < j < k$, $e'_i > e'_j < e'_k$ and $e'_i > e'_k$.
Then by \eqref{alphaprop2}, $e_j'=e_j$ and $e_k'=e_k$.  But, since $e$ avoids 201, $e'_i \neq e_i$.
Then, by definition of $\alpha$, there is an $s \in [i-1]$ such that $e'_i=e_s$.  But then $e_se_je_k$ forms a 201 pattern in $e$, a contradiction.

For the converse, let $e \in \I_n(100,210)$ and let $e'= \beta(e)$.  Suppose there is $i<j<k$ such that $e_i'e_j'e_k'$ form a 201.  We show $e$ must also contain 201.
By definition of $\beta$, since $e_i'>e_j'$, $e_i=e_i'$.  Also, there exists $s:  j \leq s \leq n$ such that $e_j'=e_s$ and $t:  k \leq t \leq n$ such that $e_k'=e_t$.  Since $e$ avoids 210,  it must be that $s < t$.  But then $e_ie_se_t$ is a 201 in $e$.

\end{proof}

\subsubsection{{\bf 2958D}: Avoiding $e_j \leq e_k$ and $e_i > e_k$}

\begin{theorem}
The patterns $(>,\geq,-)$ and $(-,\leq,>)$, defining classes {\rm 2958B} and {\rm 2958D} respectively, are Wilf-equivalent.
\label{2958BD}
\end{theorem}

\begin{proof}
The inversion sequences in classes 2958B and 2958D are defined by $\I_n(210,100)$ and $\I_n(201,100)$, respectively.
It was shown in \cite{paisI}, Theorem 5, that the following gives a bijection from
$\I_n(210)$ to $\I_n(201)$.  It can be checked that this mapping preserves 100-avoidance and therefore restricts to a bijection
from   $\I_n(210,100)$ to $\I_n(201,100)$.

Given $e \in \I_n(210)$, define $f \in \I_n(201)$ as follows.
Let $e_{a_1} \leq e_{a_2} \leq \cdots \leq e_{a_t}$ be the sequence of  weak left-to-right maxima of $e$ and let
$e_{b_1},  e_{b_2},  \ldots, e_{b_{n-t}}$  be the subsequence of remaining elements of $e$.  Since 
$e$ avoids both 210 and 100, $e_{b_1} < e_{b_2} < \cdots < e_{b_{n-t}}$.

For $i= 1, \ldots, t$, set $f_{a_i}=e_{a_i}$.  For each $j=1,2, \ldots, n-t$, we extract an element of the multiset $B=\{e_{b_1}, e_{b_2},\ldots, e_{b_{n-t}}\}$ and assign it to $f_{b_1}, f_{b_2},\ldots, f_{b_{n-t}}$ as follows: 
\[
f_{b_j} = {\rm max}\{k \ | \ k \in B - \{f_{b_1}, f_{b_2},\ldots, f_{b_{j-1}}\} \ {\rm and} \
k < {\rm max}(e_1, \ldots, e_{{b_j}-1}) \}.
\] \end{proof}
This is the same mapping that was used in Section \ref{section:772} to show that the patterns 772A and 772B are Wilf-equivalent.  Note that
\[
{\rm 772B} : \I_n(210,110,100,000) \subseteq  \I_n(210,100) : {\rm 2958B};
\]
\[
{\rm 772A} : \I_n(201,101,100,000) \subseteq  \I_n(201,100) : {\rm 2958D}.
\]

By the Wilf-equivalencies between classes B, C, and D, we know that all of 2958(A,B,C,D) are enumerated by the semi-Baxter numbers.

\subsection{Classes 3207(A,B):  $\I_n(101)$, $\I_n(110)$ }
\label{section:3207}

It was shown by Corteel, et al.\ \cite{paisI} that both $\I_n(<,-,=)=\I_n(101)$ and $\I_n(=,>,-)=\I_n(110)$ are counted by the sequence A113227 in the OEIS \cite{Sloane}, where it is said to count  $\Sn_n($1\underline{23}4$)$.
$\Sn_n($1\underline{23}4$)$ is the set of permutations with no $i < j < k < \ell$ such that
$\pi_i < \pi_j < \pi_k < \pi_{\ell}$ and $k=j+1$.

 It was proven by David Callan in \cite{callan10}, that $\Sn_n($1\underline{23}4$)$ is in bijection with increasing ordered trees with $n+1$ vertices whose leaves, taken in preorder, are also increasing.  He showed that if $u_{n,k}$ is the number of such trees with $n+1$ vertices in which the root has $k$ children then
\begin{equation}
u_{n,k} = u_{n-1,k-1} + k \sum_{j=k}^{n-1} u_{n-1,j}
\label{callan}
\end{equation}
with initial conditions $u_{0,0}=1$ and $u_{n,k}=0$ if $k > n$, or $n>0$ and $k=0$.

It was shown by Corteel, et al.\ \cite{paisI} that the number of $e \in \I_n(101)$ with exactly $k$ zeros is $u_{n,k}$, as is
the number of $e \in \I_n(110)$ with exactly $k$ zeros.
As a consequence, $101$ and $110$ are Wilf-equivalent and both avoidance sets are counted by A113227.

\subsection{Class 3720:  Quadrant Marked Mesh Patterns}
\label{section:3720}

In this section we prove that $\I_n(>,\neq,>)$ is counted by the sequence A212198 in the OEIS \cite{Sloane} where it is said to count permutations avoiding a particular marked mesh pattern.

Kitaev and Remmel \cite{KitaevRemmel} introduced the idea of quadrant marked mesh patterns, a definition of which is given below.

\begin{definition}
Let $\pi=\pi_1\pi_2 \ldots \pi_n \in S_n$.  Consider the graph of $\pi$, $G(\pi)$, consisting of the points $(i,\pi_i)$ for all $i \in [n]$.  The entry $\pi_i$ is said to \emph{match} the quadrant marked mesh pattern $\MMP(a,b,c,d)$, if in $G(\pi)$ there are at least $a$ points to the northeast of $(i,\pi_i)$, at least $b$ points to the northwest of $(i,\pi_i)$, at least $c$ points to the southwest of $(i,\pi_i)$, and at least $d$ points to the southeast of $(i,\pi_i)$. The order in which we consider quadrants proceeds counterclockwise, beginning with the top right quadrant.
\end{definition}

Let $S_n(\MMP(a,b,c,d))$ denote the set of permutations of length $n$ where no $\pi_i$ matches $\MMP(a,b,c,d)$.  We will prove that $|\Sn_n(\MMP(0,2,0,2))| = |\I_n(>, \neq ,>)|$ for all $n$. 
By symmetry established by Kitaev and Remmel \cite{KitaevRemmel}, this implies that $|\I_n(>, \neq ,>)| = |\Sn_n(\MMP(2,\allowbreak0,2,0))|$.  In our proof, we make use of the bijection  $\phi$ from Section \ref{section:1806}, whose definition was given in Definition \ref{def phi}.  Specifically, we can prove the following:

\begin{theorem}
For all $n$, $\phi(\Sn_n(\MMP(0,2,0,2))) = \I_n(>, \neq ,>)$.
\end{theorem}

\begin{proof}
This proof is very similar to the proof of Theorem \ref{thm:1806D}.  Consider some $\pi \in \Sn_n$ such that there exists some $\pi_i$ that matches $\MMP(0,2,0,2)$. This implies that there exist indices $a<b<i<j<k$ such that $\min\{\pi_a,\pi_b\}>\pi_i>\max\{\pi_j,\pi_k\}$. We will show that there exists an occurrence of the pattern $(>, \neq, >)$ in $\phi(\pi) = e$.

We must consider two cases.  If $\pi_i\leq i$, then $j >\pi_i>\pi_{j}$ and $k>\pi_i>\pi_{k}$.  Therefore, $e_i = \pi_i-1$, $e_{j} = \pi_{j}-1$ and $e_{k} = \pi_{k} -1$.  So, since $\pi_i>\max\{\pi_{j},\pi_{k}\}$ and $\pi_{j} \neq \pi_{k}$,  $e_i, e_{j}, e_{k}$ forms an occurrence of $(>, \neq, >)$.

Now assume $\pi_i> i$.  Recall that $e_i$ is the $t$-th smallest element of $\{e_{i+1},\ldots,e_n\}$ if $\pi_i$ is the $t$-th largest element of $\{\pi_1,\pi_2,\ldots,\pi_i\}$.   Since $\pi_a,\pi_b$ are larger than and occur before $\pi_i$, we know that $t$ is at least 3.  If $e_{j'}$ and $e_{k'}$ are the two smallest distinct values in the set $\{e_{i+1},\ldots,e_n\}$, we are guaranteed that $e_i > e_{j'}$ and $e_i > e_{k'}$; so, $e_i, e_{j'}, e_{k'}$ form the pattern $(>, \neq, >)$.

For the converse, let $\pi \in \Sn_n$ such that $e= \phi(\pi)$ has indices $i < j < k$ such $e_i > e_j$, $e_i > e_k$, and $e_i \not = e_k$.  We show that $\pi$ contains an entry that matches $\MMP(0,2,0,2)$.

If $e_i \in \{e_{i+1}, \ldots, e_n\}$ then,  by definition of $\phi$, $\pi_i > i$ and $e_i$ is the $k$th smallest element of $\{e_{i+1}, \ldots, e_n\}$ for some $k \geq 2$, noting that  the distinct elements $e_j$ and $e_k$ are both less than $e_i$.  Then, again by definition of $\phi$, $\pi_i$ is the $k$th largest element of $\{\pi_1, \ldots, \pi_i\}$, so there is some $a,b < i$ with $\pi_a, \pi_b > \pi_i$.  To show that $\pi_i$ matches $\MMP(0,2,0,2)$ it remains to show there are at least two elements of $\{\pi_{i+1}, \ldots, \pi_n\}$ that are smaller than $\pi_i$.  The number of elements in $\pi$ that are larger than $\pi_i$ is $n-\pi_i < n-i$.  At least two of these are to the left of $\pi_i$ in $\pi$.  Thus at most $n-i-2$ can be in the $n-i$ positions to the right of $\pi_i$.

If $e_i  \not \in \{e_{i+1}, \ldots, e_n\}$ then,  by definition of $\phi$, $\pi_i \leq i$ and $e_i = \pi_i -1$.
Since $e_j \leq e_i \leq i < j<k$ and $e_k \leq e_i \leq i < k$, we have $e_j = \pi_j -1$ and $e_k = \pi_k -1$, so that $\pi_i\pi_j\pi_k$ has the pattern 321 or 312.  To show that $\pi_i$ matches $\MMP(0,2,0,2)$, it remains to show there is an $a<i$ and $b<i$ such that $\pi_a, \pi_b > \pi_i$.  Note that there are $n-\pi_i  \geq n-i$ elements of $\pi $ larger than $\pi_i$.  The $n-i$ positions to the right of $\pi_i$ hold at least two elements smaller that $\pi_i$.  So there must be two larger element to the left of $\pi_i$, as desired.
\end{proof}

The bijection $\phi$ turns out to be a versatile tool, giving interesting results when restricted to $\Sn_n(\MMP(k,0,k,0))$ for any positive integer $k$.  In this case, $\phi(\Sn_n(\MMP(k,0,k,0)))$ maps to inversion sequences that avoid a particular set of length $k+1$ patterns.  

\subsection{Class 5040: $n!$}
\label{section:5040}

The last equivalence class of patterns in this section is the set of those avoided by all inversion sequences.
There are 41 such patterns among our 343, including the representative below. One such class is all inversion sequences with no $i>j>k$ such that $e_i=e_j=e_k$ and $e_i \neq e_k$.
  
\section{Results about patterns whose sequences don't appear in the OEIS \cite{Sloane}}

Table \ref{questions} lists all equivalence classes of the patterns $\rho \in \{<,>,\leq,\geq,=, \neq,-\}^3$
whose avoidance sequences did not appear in the OEIS (and have subsequently been entered).  We were able to derive the avoidance sequences for a few of these patterns and prove Wilf-equivalence of some others. 
In this section we describe our results and leave identification of the avoidance sequences of the remaining patterns in Table \ref{questions} as questions for future study.

\subsection{Counting results}
\label{counting}

\subsubsection{Class 805: Avoiding $e_i \leq e_j > e_k$  and $e_i \neq e_k$}
\label{section:805}

We were able to prove that $\I_n(\leq,>,\neq)$ (A279557) is counted by a sum of Catalan numbers. Recall that $C_n = \frac{1}{n+1} {2n \choose n}$ and $C(x) = \sum_{n \geq 0} C_nx^n = \frac{1-\sqrt{1-4x}}{2x}$.

Observe that $\I_n(\leq,>,\neq)$ is the set of $e \in \I_n$ satisfying, for some $t$ with $1 < t \leq n$ and k with $t<k<n$,
\[0 = e_1 = e_2 = \ldots = e_{t-1}<e_t \leq e_k \leq e_{k+1} \leq \ldots \leq  e_n\]
with $0=e_{t+1}=\cdots=e_{k-1}$.

\begin{lemma}
Let $A(x) = \sum_{n \geq 0} |\I_n(\leq,>,\neq)|x^n$. Then $A(x) = \frac{(1-2x)(1-2x-\sqrt{1-4x})}{2x^2(1-x)}$.
\end{lemma}

\begin{proof}
For any $e = (e_1,e_2,\ldots,e_n) \in \I_n(\leq,>,\neq)$, $e$ satisfies one of the following three cases:

\begin{enumerate}
\item $e = \emptyset$ or $e$ has no maximal entries,
\item $e = (0,0,\ldots,0,e_t,e_{t+1},\ldots,e_n)$ where $t$ is the index of the last occurrence of a maximal entry of $e$, or
\item $e = (0,0,\ldots,0,e_t,\ldots,e_j,\ldots,e_n)$ where $e_t \neq 0$ and $j$ is the index of the last occurrence of a maximal entry of $e$.
\end{enumerate}

The generating function for case 1 is $(1+xA(x))$, since any inversion sequence in $\I_n(\leq,>,\neq)$ with no maximal entries can be constructed by appending a 0 to the front of an inversion sequence in $\I_{n-1}(\leq,>,\neq)$.

Let $f=(f_1,f_2,\ldots,f_k)$ be an inversion sequence that avoids $(-,>,-)$ (that is, $f$ is an inversion sequence with weakly increasing entries, as seen in Section \ref{section:429}). Recall that $|\I_n(-,>,-)| = C_n$.  Then any inversion sequence in case 2 is of the form $(0, 0, \ldots, f_1, t-1, \sigma_{t-1}(f_2,f_3,\ldots,f_k))$, where $t-1$ occurs in position $t$ (and is therefore maximal). Notice that, in order for this construction to fall in case 2, $f$ must have length greater than zero.  It follows that the generating function for case 2 is $\frac{x}{1-x}\cdot (C(x)-1)$.

Finally, case 3 consists of all inversion sequences of the form $e'.(j-1).(f+j-1)$ where $e' \in (\I_{j-1}(\leq,>,\neq)\backslash\{(0,0,\ldots,0)\})$ and $(f+j-1)$ denotes the sequence obtained by adding $j-1$ to each entry of some $f \in \I_{n-t}(-,>,-)$. The generating function for this case is $(A(x)-\frac{x}{1-x})\cdot x \cdot C(x)$.

From these three disjoint cases, we have 
\[A(x) = 1+xA(x)+\frac{x}{1-x}\cdot (C(x)-1)+(A(x)-\frac{x}{1-x})\cdot x \cdot C(x).\]

Solving for $A(x)$ gives \[A(x) = \frac{1-2x}{(1-x)(1-x-xC(x))}\]

By substituting the known expression for $C(x)$, and through conjugation and algebraic rearrangement, the result follows.
\end{proof}

\begin{theorem}
$|\I_n(\leq,>,\neq)|= C_{n+1}-\sum_{i=1}^n C_i$, where $C_n = \frac{1}{n+1} {2n \choose n}$.
\end{theorem}

\begin{proof}
Let $B_n = C_{n+1}-\sum_{i=1}^n C_i$. Then $\sum_{n \geq 0} B_nx^n = \frac{C(x)-1}{x}-\frac{C(x)-1}{1-x}$. By substituting $C(x) = \frac{1-\sqrt{1-4x}}{2x}$ and simplifying, it follows that $A(x)=B(x)$ and therefore $B_n = |\I_n(\leq,>,\neq)|$.
\end{proof}

\subsubsection{Class {\bf 1016}: Avoiding $e_i > e_j$ and $e_i \not = e_k$}
\label{section:1016}

We can show that the counting sequence for $|\I_n(>,-,\neq)|$ (A279560) is as follows.  We omit the details since we hope to find a simpler formula and nicer explanation.

\[
|\I_n(>,-,\neq)| =
{2(n-1) \choose n-1} +
\sum_{k=2}^{n-2}
\sum_{i=1}^{k-1}
\sum_{u=1}^i
\sum_{d=0}^{u-1} \frac{i-d+1}{i+1} {i+d \choose d}.
\]

\subsubsection{1079(A,B):  sum of binomial coefficients}
\label{section:1079} 

The sequences in $\I_n(>,\neq,-)$ (Class 1079A, A279561)) have a nice unimodality characterization.  They are the inversion sequences $e \in \I_n$ satisfying for some $t$:
$$e_1 \leq  e_2 \leq \ldots \leq e_t \geq e_{t+1} = e_{t+2} = \ldots = e_n.$$   From this characterization, we can show the following.

\begin{theorem}
$|\I_n(>, \neq, -)| = 1 + \sum_{i=1}^{n-1} {2i \choose i-1}$.
\end{theorem}

Additionally, we conjecture that $\I_n(<,>,\neq)$ (Class 1079B) is Wilf-equivalent, but have not proven this.

\subsubsection{4306A,B:  $\I_n(210)$, $\I_n(201)$}

The sets $\I_n(-,<,>)=\I_n(201)$ and $\I_n(>,>,-)=\I_n(210)$ were shown to be Wilf-equivalent by Corteel, et al.\ \cite{paisI} via a bijection, which we made use of in Section \ref{section:2958}. 
An alternate proof of this fact was discovered by Mansour and Shattuck \cite{Mansour}.
 A recurrence to compute $|\I_n(201)|$  was also derived in \cite{paisI} .

\subsection{Wilf-equivalence results}
\label{equivalences}

In the remainder of the section we show that the bijection $\alpha$ described in \eqref{alpha} of Section \ref{section:2958} proves Wilf-equivalence of all of the following pairs of patterns:
663A,B; 746A,B; 1833A,B;  and 1953A,B.

\subsubsection{Class {\bf 1953A}: Avoiding $e_j>e_k$ and $e_i>e_k$, and Class {\bf 1953B}: Avoiding $e_i \neq e_j \geq e_k$ and $e_i>e_k$}

\begin{theorem}
The patterns $(-,>,>)$ and $(\neq,\geq,>)$, defining classes {\rm 1953A} and {\rm 1953B} respectively, are Wilf-equivalent.
\label{thm:1953AB}
\end{theorem}

\begin{proof}
The inversion sequences in classes 1953A and 1953B are those in 
$\I_n(110,210,120)$ and $\I_n(100,210,120)$, respectively.  Notice that these are the inversion sequences in $\I_n(110,210)$ (class {\rm 2958C}) and $\I_n(100,210)$ (class {\rm 2958A}), respectively, that avoid 120. So it suffices to show that both $\alpha: \I_n(110,210) \rightarrow \I_n(100,210)$ and $\beta=\alpha^{-1}$ preserves 120-avoidance.

Suppose $e \in \I_n(110,210)$ avoids 120, but for $e'=\alpha(e)$ there exist $i < j < k$ such that $e_i'<e_j'>e_k'$ and $e_i' > e_k'$.  By \eqref{alphaprop2}, $e_k'=e_k$.  Notice that we cannot have both $e'_i=e_i$ and $e'_j=e_j$, since this would create a 120 in $e$.

Suppose first that $e_j'=e_j$.  Since $e$ avoids 120, $e_i' \neq e_i$ so, by definition of $\alpha$, there is an $s \in [i-1]$ such that $e_s=e_i'$.  But then $e_se_je_k$ forms a 120 in $e$.

So, assume that $e_i'=e_i$.  Then, since $e$ avoids 120, $e_j' \neq e_j$.  So, there must be a $t \in [j-1]$ such that $e_t=e_j'$. If $i < t < k$ then $e_ie_te_k$ is a 120 in $e$.  Otherwise, $t < i < k$ and $e_te_ie_k$ is a 210 in $e$, which is impossible.

Finally, if both $e_i' \not =e_i$ and $e_j' \not=e_j$, then let $s$ and $t$ be as above.  If $s < t$ then $e_se_te_k$ forms a 102 in $e$. Otherwise, $e_te_se_k$ forms a 210 in $e$.  Both cases lead to a contradiction.

If both $e_i' \not =e_i$ and $e_j' \not=e_j$, then let   $s \in [i-1]$ and $j \in [j-1]$  be indices such that $e'_i=e_s$ and $e'_j=e_t$.  Since $e'_i=\max\{e_1,e_2,\ldots,e_i\}$ and $e'_j=\max\{e_1,e_2,\ldots,e_j\}$, and we have $e'_i < e'_j$, it must be the case the $e_s<e_t$ and $s<t$.  Therefore  $e_se_te_k$ is an occurrence of 120 in $e$, giving a contradiction.

It follows that $\alpha$ preserves 120-avoidance.  Showing that $\beta$ preserves 120-avoidance is similar.
\end{proof}

\subsubsection{Class {\bf 1833A}: Avoiding $e_j \neq e_k$ and $e_i>e_k$, and Class {\bf 1833B}: Avoiding $e_i \neq e_j$ and $e_i>e_k$}

\begin{theorem}
The patterns $(-,\neq,>)$ and $(\neq,-,>)$, defining classes {\rm 1833A }and {\rm 1833B} respectively, are Wilf-equivalent.
\label{thm:1833AB}
\end{theorem}

\begin{proof}
The inversion sequences in classes 1833A and 1833B are those in $\I_n(110,210,120,201)$ and $\I_n(100,210,120,201)$, respectively. It was shown in Section \ref{section:2958} that the bijection $\alpha: \I_n(110,210) \rightarrow \I_n(100,210)$ preserves 201-avoidance, as does its inverse $\beta$.  It follows from Theorem \ref{thm:1953AB} that $\alpha$ and $\beta$ preserves 120-avoidance as well. So $\alpha(\I_n(-,\neq,>)) = \I_n(\neq,-,>)$.
\end{proof}

\subsubsection{Class {\bf 746A}: Avoiding $e_j > e_k$ and $e_i \geq e_k$, and Class {\bf 746B}: Avoiding $e_i \neq e_j \geq e_k$ and $e_i \geq e_k$}

\begin{theorem}
The patterns $(-,>,\geq)$ and $(\neq,\geq,\geq)$, defining classes {\rm 746A} and {\rm 746B} respectively, are Wilf-equivalent.
\label{thm:746AB}
\end{theorem}

\begin{proof}
The inversion sequences in classes 746A and 746B are those in $\I_n(110,210,120,010)$ and $\I_n(100,210,120,010)$, respectively. These are, in turn, the inversion sequences in $\I_n(110,\allowbreak210,120)$ (class {\rm 1953A}) and $\I_n(100,210,120)$ (class {\rm 1953B}), respectively, that avoid 010.
By Theorem \ref{thm:1953AB}, $\alpha: \I_n(110,210) \rightarrow \I_n(100,210)$ and $\alpha^{-1}=\beta$ preserve 120-avoidance. So it suffices to show that both  $\alpha: \I_n(110,210) \rightarrow  \I_n(100,210)$ and $\alpha^{-1}=\beta$ preserves 010-avoidance.  

Suppose $e \in \I_n(110,210)$ avoids 010, but for $e'=\alpha(e)$ there exist $i < j < k$ such that
$e_i'<e_j'>e_k'$ and $e_i'= e_k'$.  By \eqref{alphaprop2}, $e_k'=e_k$. Since $e$ avoids 010, it follows that we cannot have both $e_i'=e_i$ and $e_j'=e_j$.

Suppose first that $e_j'=e_j$.  Since $e$ avoids 010, $e_i' \neq e_i$ so, by definition of $\alpha$ there is an $s \in [i-1]$ such that $e_s=e_i'$.  But then $e_se_je_k$ forms a 010 in $e$.

So, assume that $e_i'=e_i$.  Then, since $e$ avoids 010, $e_j' \neq e_j$.  So, there must be a $t \in [j-1]$ such that $e_t=e_j'$. If $i < t < k$ then  $e_ie_te_k$ is a 010 in $e$.  Otherwise, $t < i < k$ and $e_te_ie_k$ is a 100 in $e$, which is impossible.

If both $e_i' \not =e_i$ and $e_j' \not=e_j$, then let   $s \in [i-1]$ and $t \in [j-1]$  be indices such that $e'_i=e_s$ and $e'_j=e_t$.  Since $e'_i=\max\{e_1,e_2,\ldots,e_i\}$ and $e'_j=\max\{e_1,e_2,\ldots,e_j\}$, and we have $e'_i < e'_j$, it must be the case the $e_s<e_t$ and $s<t$.  Additionally, $e_s=e'_i=e'_k=e_k$.  Therefore  $e_se_te_k$ is an occurrence of 010 in $e$, giving a contradiction.

It follows that $\alpha$ preserves 010-avoidance.  Showing that $\beta$ preserves 010-avoidance is similar.
\end{proof}

\subsubsection{Class {\bf 663A}: Avoiding $e_j \neq e_k$ and $ e_i \geq e_k$, and Class {\bf 663B}: Avoiding $e_i \neq e_j$ and $e_i  \geq e_k$}

\begin{theorem}
The patterns $(-,\neq,\geq)$ and $(\neq,-,\geq)$, defining classes {\rm 663A} and {\rm 6633B} respectively, are Wilf-equivalent.
\label{thm:663AB}
\end{theorem}

\begin{proof}
The inversion sequences in classes 663A and 663B are those in $\I_n(110,210,120,201,\allowbreak010)$ and $\I_n(100,210,120,201,010)$, respectively. It was shown in Section \ref{section:2958} that the bijection $\alpha: \I_n(110,210) \rightarrow \I_n(100,210)$ preserves 201-avoidance, as does its inverse. It follows from Theorem \ref{thm:1953AB} that $\alpha, \beta$ preserve 120-avoidance and from Theorem \ref{thm:746AB}  that they preserve 010-avoidance as well. So $\alpha(\I_n(-,\neq,\geq) )= \I_n(\neq,-,\geq)$.
\end{proof}

\section{Concluding remarks}



The results in this work demonstrate that inversion sequences avoiding a triple of relations have connections to a vast array of pattern-avoiding classes and combinatorial sequences.  
Moreoever, as a pattern to avoid, a triple of relations can be more expressive than a single word of length 3, even though it is always equivalent to an avoidance {\it set} of words of length 3.  

On the other hand, $\I_n(001,210)$ is not the avoidance set of any triple of relations considered in this paper.  In ongoing work, we construct and examine the partially ordered set whose elements are the avoidance sets $\I_n(S)$,  ordered by inclusion, where $$S \subseteq \{000,001,010,100,011,101,110,123,132,213,231,312,321\}.$$

Several interesting  questions remain for future work, such those highlighted in Table~\ref{other triples} with a ``no'' in column 3.  For instance, is $\I_n(>,-,>)$ equinumerous with $\Sn_n(4231,42513)$ and is $\I_n(>,-,\leq)$ equinumerous with $\Sn_n(4123,4132,4213)$?  Are $\I_n(=,\geq,-)$  and  $\I_n(\geq,=,-)$ counted by the Bell numbers?

 Another fascinating open question is whether there is a bijective mapping to relate the Baxter permutations in $\Sn_n$ to $\I_n(\geq, \geq, >)$. There are also a number of lingering open enumeration problems: can enumeration formulas be found for some of the avoidance sets in Table \ref{questions}, such as $\I_n(010)$, $\I_n(100)$, $\I_n(120)$,  or $\I_n(201)=\I_n(210)$?

Pattern avoidance can be studied in more general $s$-inversion sequences.
For a given sequence $s=(s_1, s_2, \ldots, s_n)$ of positive integers, the $s$-inversion sequences $\I_n^{(s)}$ are defined by
\[
\I_n^{(s)} = \{ (e_1,e_2, \ldots, e_n)\in Z^n \ | \  0 \leq e_i < s_i, \ 1 \leq i \leq n\}.
\]
When $s=(1,2, \ldots, n)$, these are the usual inversion sequences.  For $s=(2,4, \ldots, 2n)$, there is a statistics-preserving bijection between $\I_n^{(s)}$ and signed permutations of $[n]$.  See \cite{survey} for other interesting families of $s$-inversion sequences, and for an overview of  their combinatorial and geometric relationships to $s$-lecture hall partitions and cones.

\section{Acknowledgements}

Thanks to Sylvie Corteel for helpful discussions about several of the sequences herein.  Thanks to Michael Weselcouch for independently verifying our computations of the avoidance sequences for the 343 patterns. Thanks to the referee of this paper for a careful reading of the manuscript and many helpful suggestions.
Thanks to Michael Albert for making his PermLab software freely available \cite{PermLab1.0}.
Thanks to the Simons Foundation for Collaboration Grant No. 244963 to the second author which supported the travel of the first author for collaboration.  Some of this 
material is based upon work supported by the National Science Foundation under Grant No. 1440140, while the second author was in residence at the Mathematical Sciences Research Institute in Berkeley, California, during the fall semester of 2017.

We once again extend sincere thanks to Neil Sloane and the OEIS Foundation, Inc. The On-Line Encyclopedia of Integer Sequences \cite{Sloane} was an invaluable resource for this project.
 
\bibliographystyle{jis}
\bibliography{pais2}

\end{document}